\documentclass[12pt]{article}

\usepackage{graphicx} 

\usepackage{amsfonts}
\usepackage{amsmath}
\usepackage{amssymb}
\usepackage{amsthm}
\usepackage{enumitem}
\usepackage[left=3cm,right=3cm,top=3.5cm,bottom=3.5cm]{geometry}

\usepackage{tikz-cd}
\usepackage{arydshln}
\usepackage{tikz}

\usepackage[colorlinks=true]{hyperref}

\title{Holomorphic geometric structures on Hopf manifolds}
\author{Matthieu Madera}
\date{}

\newtheorem{Theorem}{Theorem}
\newtheorem{Lemma}{Lemma}[subsection]
\newtheorem{Proposition}{Proposition}[subsection]

\newtheorem{Remark}{Remark}[subsection]
\newtheorem{Definition}{Definition}[subsection]
\newtheorem{Corollary}{Corollary}[subsection]
\newtheorem{Example}{Example}[subsection]

\DeclareMathOperator\im{im}

\DeclareMathOperator\Hom{Hom}
\DeclareMathOperator\Vect{Vect}

\begin{document}

\maketitle

\begin{abstract}
    We construct integrable holomorphic $G$-structures and flat holomorphic Cartan geometries on every complex Hopf manifold, without using the normal forms given by the Poincaré-Dulac Theorem. We provide a new proof of the latter using charts adapted with the geometric structures.
\end{abstract}

\tableofcontents

\section{Introduction}\label{Intro}

We study complex Hopf manifolds. This is a class of compact complex manifolds of any dimension $n\geq 2$ that do not admit kählerian metrics. As real manifolds of dimension $2n$, they are in fact isomorphic to the product $\mathbb S^1\times\mathbb S^{2n-1}$ and hence do not have the topology of symplectic manifolds.

The complex structures on $\mathbb S^1\times\mathbb S^{2n-1}$, turning it into a Hopf manifold, are classified with the help of a famous dynamics theorem, the \textbf{Poincaré-Dulac Theorem} (\cite{A88}, p.194). This theorem provides normal forms for Hopf manifolds.

The aim of this paper is to construct holomorphic geometric structures on Hopf manifolds of any dimension, without using the existence of normal forms. Throughout the paper, the term ``geometric structure" can stand for several concepts. The main point of view is the one of \textbf{$G$-structures} of higher order (Section \ref{G structures}), but we will also discuss about Gromov geometric structures (Section \ref{G structures}) and \textbf{Cartan geometries} (Section \ref{Sect Cart Geo}). 

The reason why we do not want to take advantage of the Poincaré-Dulac Theorem is that one of the goals of the paper is to show the close relationship between the existence of a normal form for a given Hopf manifold and the existence of a holomorphic geometric structure on it, which can be seen as an integrable $G$-structure, or as a flat Cartan geometry.

In that sense, without any need of normal forms, we construct holomorphic $G$-structures of higher order on every Hopf manifold and we show that they induce holomorphic Cartan geometries. Studying these structures, we prove that some $G$-structures are integrable (Definition \ref{Integrability}) and that the induced Cartan geometry is flat. We infer that there exist special charts adapted with the integrable structure, providing a normal form for the Hopf manifold, and proving the Poincaré-Dulac Theorem.

\subsection{Introduction to Hopf manifolds}\label{History}

A (primary) Hopf manifold is a compact complex manifold of dimension $n\geq 2$, having its universal covering space biholomorphic to $\mathbb C^n\setminus \{0\}$ and having an infinite cyclic fundamental group. As an example, Hopf \cite{Ho48} first considered actions of $\mathbb Z$ on $\mathbb C^2\setminus\{0\}$ of the type 
$$
\begin{matrix}
    \mathbb Z\times \mathbb C^2\setminus \{0\} & \longrightarrow & \mathbb C^2\setminus \{0\} \\
    (k,z) & \longmapsto & (\frac{1}{2})^kz
\end{matrix}
$$
This action of $\mathbb Z$ preserves the standard torsion-free flat affine connection of $\mathbb C^2\setminus\{0\}$, and therefore, the Hopf surface defined as the quotient of $\mathbb C^2\setminus \{0\}$ under this action carries naturally a torsion-free flat affine connection. More generally, any Hopf manifold obtained as the quotient of $\mathbb C^n\setminus\{0\}$ by the action generated by an invertible linear map is naturally endowed with a torsion-free flat affine connection. Those Hopf manifolds are called \textbf{linear Hopf manifolds}. 

A torsion-free affine connection can be seen from two points of view. First, we can see it as a Cartan geometry with model the affine group of $\mathbb C^n$ and its closed subgroup the general linear group (see Section \ref{Sect Cart Geo}). On the other hand, it can be seen as a $GL(n,\mathbb C)$-structure of order $2$ (Example \ref{example G struc}).

Borcea (\cite{Bo81}, p.6) proved that the fundamental group of a Hopf manifold is generated by an element that has the topological properties of what is called a contraction of $\mathbb C^n$ at $0$ (Section \ref{Resonances}). Wehler (\cite{W81} for surfaces) and Haefliger (\cite{Ha85}, Appendix) did work on the Kuranishi family of such a Hopf manifold. They took advantage of a well-known theorem of holomorphic dynamics, the Poincaré-Dulac Theorem (\cite{A88}, p.194). A contraction can be put, with a local holomorphic change of coordinates, under a normal form. This normal form is a resonant polynomial map (\cite{A88}, p.194, Definition \ref{def reson} for the definition of resonance). It appears that the space of normal forms of the germ of the contraction is closely related with the local deformation space of the complex structure of the corresponding Hopf manifold.

Using the Poincaré-Dulac theorem, McKay and Pokrovskiy (\cite{MP2010}) proved that any Hopf surface carries locally homogeneous structures. That is, the complex structure of a Hopf surface is always induced by an atlas of local charts with values in a suitable complex homogeneous space $G/H$. They classify all possible structures with model $\mathcal{O}(n)$ (the line bundle of degree $n$ over the Riemann sphere $\mathbb P^1(\mathbb C)$) on any Hopf surface, depending on the normal form of the contraction. They showed that every Hopf surface carries some flat Cartan geometries. These geometries generalize torsion-free flat affine connections that naturally exist on linear Hopf surfaces.

On the other hand, some authors did work on Hopf manifolds without using the Poincaré-Dulac Theorem, i.e. without assuming that the contraction defining the Hopf manifold is under a polynomial normal form. In this context, Mall (\cite{Ma96}) defined the notion of Mall bundles on Hopf manifolds. A Mall bundle on a Hopf manifold is a holomorphic vector bundle with a trivial pullback on the universal covering space $\mathbb C^n\setminus\{0\}$. Mall was able to give a full description of the cohomology of such a Mall bundle, using the fact that the linear operator given by the contraction on space of sections is compact. We will also use this property in Section \ref{geo str ord 1}, more precisely in the proof of Theorem \ref{th 1}. 

Recently, keeping this idea of considering contractions under arbitrary forms, and using the link between the existence of an affine connection and cohomology of some tensor vector bundles, which are Mall bundles, Ornea and Verbitsky (\cite{OV2024}) proved that a Hopf manifold given by a non-resonant contraction (Definition \ref{def reson}) carries a unique affine connection, which is \textit{a posteriori} torsion-free and flat. The contraction then preserves an affine atlas on $\mathbb C^n$, and an affine local system of coordinates around $0$ linearises the contraction. This is a way, using the geometry of Hopf manifolds, to prove the non-resonant case of the Poincaré-Dulac Theorem, also known as the Poincaré linearization Theorem (\cite{A88}, p.193). In \cite{Mc2016} (p.21, lemma 14.2), McKay proved also that a Hopf manifold admitting a holomorphic affine connection (not necessarily torsion-free, nor flat) is isomorphic to a linear Hopf manifold, an isomorphism being given by the exponential chart at $0$ of the pulled back affine connection extended to $\mathbb C^n$.

During the development of this work, we had the opportunity to discuss with Paul Boureau, who also presents in \cite{Bou2025} new results on holomorphic geometric structures on Hopf manifolds, using a different approach. Indeed, he offers a re-examination of Poincaré-Dulac normal forms and the study of the group of sub-resonant polynomials, and uses these normal forms to exhibit locally homogeneous geometric structures on Hopf manifolds of any dimension. This result generalizes those of McKay and Pokrovskiy, who dealt in \cite{MP2010} with the case of Hopf surfaces.

\subsection{Description of the results}\label{Description}

We propose a generalization of the paper of Ornea and Verbitsky \cite{OV2024}. Given a contraction, we consider a group $G_\beta^r$ of sub-resonant polynomial transformations of maximal degree $r$ of $\mathbb C^n$ (Section \ref{group}). This group depends only on the eigenvalues of the differential at $0$ of the contraction. The parameter $\beta$ stands for the collection of eigenvalues and their multiplicity, and the parameter $r$ is the maximal length of the resonances among the eigenvalues (Section \ref{Resonances}).

We construct a $G_\beta^1$-structure of order $1$ on the Hopf manifold, where the group $G^1_\beta$ is given by the jets of order $1$ at $0$ of the elements of the group $G_\beta^r$ (Section \ref{G structures} for the definition of $G$-structure, Section \ref{geo str ord 1} Theorem \ref{th 1} for the construction). More specifically, we prove that the contraction preserves an entanglement of vector sub-bundles of the holomorphic tangent bundle $T\mathbb C^n$. As a consequence, the Hopf manifold carries various sub-bundles of its holomorphic tangent bundle. For example, carrying a $G^1_\beta$-structure implies that there is a complete flag of vector sub-bundle of the holomorphic tangent bundle. This is the first step in the construction of $G$-structures of higher order on Hopf manifolds.

Given such a structure of order $1$, using cohomology arguments on Mall bundles, we prove that the structure extends uniquely to a $G^2_\beta$-structure of order $2$, where the group $G^2_\beta$ is given by the jets of order $2$ of elements of $G_\beta^r$ (Section \ref{geo str ord sup}, Theorem \ref{recurrence pour les structures} and Section \ref{geo str on surf} for surfaces). 

This is a generalization of the cohomological technique used by Ornea and Verbitsky in \cite{OV2024}. Notice that the extension of our $G^1_\beta$-structure of order $1$ is a generalization of an affine connection, but in the latter case, there is no need to reduce the group of the order $1$ frame bundle. A torsion-free affine connection is indeed a $GL(n,\mathbb C)$-structure of order $2$ (\cite{Du11}, p.68). In the non-resonant case, the holomorphic tangent bundle $TM$ of the Hopf manifold admits a holomorphic affine connection because the space $H^1(M,TM^*\otimes TM^* \otimes TM)$ is trivial. A torsion-free affine connection is an equivariant way of extending all the jets of order $1$ of charts to jets of order $2$. In the resonant case, we have to consider cohomology in some vector bundles taking into account the resonances and we show that the first cohomology space of this vector bundle is trivial. With these considerations, we show that we can extend some jets of order $1$ of charts, namely those given by a $G^1_\beta$-structure of order $1$, into equivalent classes of jets of order $2$. The latter equivalent classes are precisely given by what we call the sub-resonant monomials (Section \ref{geo str ord sup}). 

By induction, we can prove that every Hopf manifold admits special $G$-structures of every order (Section \ref{geo str ord sup}, Theorem \ref{recurrence pour les structures}), except for surfaces (Section \ref{geo str on surf}), where we prove that, \textit{a priori}, the structure extends up to order $r+2$. 

Since the group $G_\beta^r$ is a group of global polynomial transformations of degree at most $r$, the extension from the $G^r_\beta$-structure of order $r$ to a structure of order $r+1$ is realized by an equivariant way of extending jets of order $r$ of the holomorphic charts given by the structure into jets of order $r+1$ (see Section \ref{cartan geo}). Using this equivariant section, we can pull back the canonical $1$-form of the frame bundle of order $r+1$ (see \cite{Ko61} and Section \ref{subsection canonical form}) on the structure of order $r$, and we prove in Theorem \ref{on a une geometrie de Cartan}, that it leads to a holomorphic Cartan geometry (see \cite{Sh2000} and Section \ref{Sect Cart Geo} for the definitions), with model defined in Section \ref{new lie group}. The strategy of the proof is the following. The pullback of the canonical form induces a horizontal distribution of the frame bundle of order $r$ defined over the reduction (see Section \ref{section geom premi} for the definition of the frame bundle of higher order). The defect of this distribution to be actually tangent to the reduction is measured by a holomorphic section of a Mall bundle. We prove that such a section, being invariant by the contraction and containing the data of the resonances, has to vanish.

The fact that the previous horizontal distribution is tangent to the reduction is a strong tool that will help us to prove that the $G^r_\beta$-structure of order $r+1$ turns out to be integrable (see Definition \ref{Integrability}). We will use a Frobenius-type theorem, given by Benoist in his paper bringing his point of view on the demonstration of Gromov's open-dense orbit Theorem (\cite{Be97}, Theorem p.12, Section \ref{G structures}). The existence of an integrability atlas is equivalent to finding solutions of a partial derivatives relation. Here, the completeness (Definition \ref{def comp et cons}) of this relation is given by the equivariant section mentioned above. Roughly speaking, it corresponds to the rigidity in the sense of Gromov geometric structures (see Definition \ref{def rigidity}). The consistence is obtained by considering the extension of the structure of order $r$ up to order $r+2$ which, by the same argument as before, produces a Cartan geometry on the reduction of order $r+1$.

We notice that, in the case of Hopf surfaces, the flat Cartan geometries we construct are part of McKay and Pokrovskiy's classification of $\mathcal O(n)$-structures on Hopf surfaces in \cite{MP2010}. What is new here is that we construct those structures without using the normal form of the contraction.

The existence of this integrable structure, and \textit{a fortiori}, a flat Cartan geometry (Theorem \ref{flatness of the cartan geometry}), is equivalent to the existence of special charts adapted to the geometric structure, which provide new coordinates where the contraction has a polynomial normal form. Therefore, we infer a proof of the Poincaré-Dulac Theorem in this context of global contractions, generalizing the proof given by Ornea and Verbitsky in \cite{OV2024} for the non-resonant case (see Theorem \ref{th PoincareDulac}). 

For a Hopf manifold given by a contraction in normal form, the geometric structures we consider arise naturally, as in the case of linear Hopf manifolds. What we prove in this paper is the converse. A contraction has to preserve some integrable $G$-structures, or equivalently some flat Cartan geometries, and special charts for those structures provide the change of coordinates putting the contraction under a normal form.

The main steps of the paper are the following. In Section \ref{geo str ord 1}, we prove the existence of structures of order $1$. In Section \ref{geo str ord sup}, the induction step in the construction of geometric structures of higher order is explained. In Section \ref{Poincare Dulac} we explain how a Cartan geometry is induced by the $G$-structures we consider (Section \ref{cartan geo}), we prove the integrability of the structure of order $r+1$ and the flatness of the Cartan geometry (Section \ref{integrability}) and we deduce a proof of the Poincaré-Dulac Theorem (Section \ref{subsect Poincare Dulac}).

\section{Geometric preliminaries}\label{section geom premi}

In this section, $M$ is a complex manifold of dimension $n$ and $r$ is a positive integer. 

We denote by $\mathcal{D}^r(\mathbb C^n)$ the set of equivalence classes of germs of biholomorphic maps $(\mathbb C^n,0)\to (\mathbb C^n,0)$ under the relation ``having the same Taylor polynomial of order $r$ at $0$". Endowed with the truncated composition, $\mathcal{D}^r(\mathbb C^n)$ is a complex Lie group (\cite{Be97}). We will denote by $\mathfrak d^r(\mathbb C^n)$ its Lie algebra. By convention, $\mathfrak d^0(\mathbb C^n)$ is the trivial Lie algebra. As an example, the group $\mathcal{D}^1(\mathbb C^n)$ is isomorphic to the general linear group $GL(n,\mathbb C)$. For $r\geq s \geq 0$, we denote by $\rho_{r,s}:\mathcal D^r(\mathbb C^n)\to \mathcal{D}^s(\mathbb C^n)$ the standard projection, with the convention $\mathcal D^0(\mathbb C^n)=\{e\}$. This is a complex Lie groups homomorphism.

On the complex manifold $M$, we consider the set of equivalence classes of germs of biholomorphic maps $(\mathbb C^n,0)\to (M,m)$ under the relation ``having the same Taylor polynomial of order $r$ at $0$ in any fixed holomorphic chart in a neighborhood of $m\in M$". The collection of all those sets patches to form a principal $\mathcal{D}^r(\mathbb C^n)$ bundle called the order $r$ frame bundle over $M$, with total space denoted by $\mathcal R^r(M)$(\cite{Be97}). Notice that $\mathcal R^1(M)$ is isomorphic to the bundle of linear frames of the holomorphic tangent bundle $TM$ of $M$. For $r\geq s \geq 0$, we denote by $\pi_{r,s}:\mathcal R^r(M)\to \mathcal{R}^s(M)$ the standard projection, with the convention $\mathcal R^0(M)=M$. This is a homomorphism of principal bundles, with underlying Lie group homomorphism $\rho_{r,s}$.

\subsection{Canonical forms on frame bundles}\label{subsection canonical form}

We recall here the definition and the main properties of the canonical forms on the frame bundles over a complex manifold $M$.

First, we notice that the translations of $\mathbb C^n$ give rise to a global section of all the frame bundles of $\mathbb C^n$, and hence global trivializations
$$
\begin{matrix}
    \xi_r & : & \mathcal R^r(\mathbb C^n) & \longrightarrow & \mathbb  C^n\times \mathcal{D}^r(\mathbb C^n) \\
    && j^r_0f & \longmapsto & (f(0),j^r_0(\theta_{-f(0)}f))
\end{matrix}
$$
where $\theta_z$ is the translation of vector $z\in\mathbb C^n$.

We also notice that if $U$ and $V$ are complex manifolds, and if $f:U\to V$ is a holomorphic map, then it induces a holomorphic map 
$$
\begin{matrix}
    j^r f & : & \mathcal R^r(U) &\longrightarrow & \mathcal R^r(V) \\
    && j^r_0g &\longmapsto & j^r_0(fg).
\end{matrix}
$$

We define (\cite{Mo2006}, p.14), for $r\geq 0$, the following holomorphic $1$-form on $\mathcal R^{r+1}(M)$, called the canonical form of the $r+1$ order frame bundle 
$$
\begin{matrix}
    \chi_{r+1} & : & T\mathcal R^{r+1}(M) & \longrightarrow & T_0\mathbb C^n\times \mathfrak d^{r}(\mathbb C^n) \\
    & & (j^{r+1}_0f,X) & \longmapsto & \mathrm d_{j^r_0id}\xi_r((\mathrm{d}_{j^r_0id}j^rf)^{-1}(d_{j^{r+1}_0f}\pi_{r+1,r}(X))).
\end{matrix}
$$

This canonical form illustrates the following fact. For $r\geq 0$, an element $j^{r+1}_0f$ of $\mathcal R^{r+1}(M)$ can be seen as a horizontal plane for the bundle $\pi_{r,0}:\mathcal{R}^r(M)\to M$ at $j^r_0f$. Together with the standard trivialization of the vector bundle of vertical tangent vectors of $\mathcal{R}^r(M)$ ($\pi_{r,0}$ is a principal bundle) and the fact that $j^1_0f$ gives a linear frame of $T_{f(0)}M$, $j^{r+1}_0f$ is naturally seen as a frame of order $1$ of $\mathcal{R}^r(M)$ at $j^r_0f$. The canonical form at a vector $X$ is obtained by reading the coordinates of the projection of $X$ in this frame. Notice that for $r=0$ the form on $\mathcal R^1(M)$ is given by reading the coordinates of a tangent vector of $M$ in the given frame.

We will now give some properties of the canonical form. For $r\geq 0$, we consider the following group representation
$$
\begin{matrix}
    \underline{Ad}^{r+1} & : &  \mathcal{D}^{r+1}(\mathbb C^n) & \longrightarrow & GL(T_{j^r_0id}\mathcal R^r(\mathbb C^n)) \\
    && j^{r+1}_0\phi & \longmapsto & \underline{Ad}^{r+1}_{j^{r+1}_0\phi}:=\mathrm{d}_{j^r_0id}(j^r_0f\mapsto j^r_0(\phi f\phi^{-1})).
\end{matrix}
$$
Using $\xi_r$, we can consider that $\underline{Ad}^{r+1}$ takes values in $GL(T_0\mathbb C^n\times \mathfrak d^{r}(\mathbb C^n))$. Using this representation, we have the following proposition.

\begin{Proposition}\label{properties cano form}
    (\cite{Mo2006}, p.14) For $r\geq 0$, the canonical form $\chi_{r+1}$ satisfies the following properties:
    \begin{enumerate}
        \item $\forall v\in \mathfrak d^{r+1}(\mathbb C^n)$, $\chi_{r+1}(X_v)\equiv (0,\mathrm d_{e}\rho_{r+1,r} (v))$, where $X_v$ is the fundamental vector field of $\mathcal R^{r+1}(M)$ generated by $v$ ;
        \item $\forall j^{r+1}_0\phi\in\mathcal D^{r+1}(\mathbb C^n)$, $\mathcal{R}_{j^{r+1}_0\phi}^*\chi_{r+1}=\underline{Ad}^{r+1}_{j^{r+1}_0\phi^{-1}}\chi_{r+1}$.
    \end{enumerate}
\end{Proposition}

\subsection{Holomorphic G-structures on manifolds}\label{G structures}

Let $G$ be a complex Lie subgroup of $\mathcal D^r(\mathbb C^n)$. We recall the definition of a $G$-structure of order $r$ on the manifold $M$.

\begin{Definition}\label{Def G str}
    (\cite{Vi91}, p.152 Example 5) A $G$-structure of order $r$ on $M$ is a principal $G$ sub-bundle of the order $r$ frame bundle of $M$. Equivalently, it is a global section of the quotient bundle $\mathcal R^r(M)/G$.
\end{Definition}

Roughly speaking, a $G$-structure of order $r$ on a manifold is a system of local sections of the $r$ order frame bundle ($r$-frames) such that, on the overlaps, the changes of frames are everywhere elements of $G$.

\begin{Example}
    (\cite{St64}, p.315) On $M=\mathbb C^n$, the standard global coordinates define a trivialization of the frame bundle of any order over $\mathbb C^n$. For every $r$, this trivialization is a $\{e\}$-structure of order $r$, where $\{e\}$ is the trivial group. Moreover, for every Lie subgroup $G$ of $\mathcal D^r(\mathbb C^n)$, this $\{e\}$-structure induces a $G$-structure of order $r$ on $\mathbb C^n$. We call this $G$-structure the standard $G$-structure of $\mathbb C^n$.
\end{Example}

\begin{Definition}\label{Integrability}
    (\cite{Vi91}, p.187) A $G$-structure of order $r$ over M is said to be integrable if it is locally isomorphic to the standard $G$-structure of $\mathbb C^n$, i.e. if there exists a holomorphic atlas on $M$ such that on every open set of this atlas, the chart sends the standard local section of $\mathcal{R}^r(\mathbb C^n)\to\mathbb C^n$ to a local section of the $G$-structure over $M$.
\end{Definition}

\begin{Example}\label{example G struc}
    \begin{enumerate}
        \item (\cite{Ko95}, p.13 Example 2.11) An absolute parallelism on $M$ is a trivialization of the tangent bundle of $M$. Equivalently, this is a global section of the order $1$ frame bundle $\mathcal R^1(M)$. Therefore, it is a $\{e\}$-structure of order $1$. The parallelism is given by $n$ holomorphic vector fields, pointwise linearly independents. The structure is therefore integrable if and only if the vector fields commute two by two.
        \item (\cite{Vi91}, p.185 Example 2) A distribution of dimension $d$ in the tangent bundle is a $G$-structure of order $1$ with the group $G$ defined as the subgroup of $GL(n,\mathbb C)$ composed of the elements preserving a $d$-dimensional vector subspace of $\mathbb C^n$. By the classical Frobenius Theorem, a distribution is integrable (and produces a foliation) if and only if it is involutive under the Lie bracket.
        \item (\cite{Du11}, p.68) The group $GL(n,\mathbb C)$ is naturally a Lie subgroup of $\mathcal D^2(\mathbb C^n)$. A $Gl(n,\mathbb C)$-structure of order $2$ on $M$ is a torsion-free holomorphic affine connection. The structure is integrable if and only if the curvature tensor vanishes. 
    \end{enumerate}
\end{Example}

Another notion is the notion of rigidity in the context of Gromov geometric structures. We recall here the definitions (\cite{Gr88}, §0, \cite{Be97}, Section 1.2 p.4).

Let $Z$ be a complex manifold endowed with a holomorphic action of $\mathcal{D}^r(\mathbb C^n)$.

\begin{Definition}\label{Gromov geo str}
    (\cite{Gr88} §0, \cite{Be97} p.4) A Gromov geometric structure of type $Z$ and of order $r$ on $M$ is a holomorphic $\mathcal{D}^r(\mathbb C^n)$-equivariant map $g:\mathcal{R}^r(M)\to Z$.
\end{Definition}

\begin{Remark}\label{rmk algebraic}
    In \cite{Be97}, the manifold $Z$ and its action of $\mathcal D^r(\mathbb C^n)$ are supposed to be algebraic. This assumption is fundamental for the open-dense orbit Theorem, but we will not use it here. Nevertheless, we notice that every Gromov geometric structure we consider in this paper is actually of algebraic type, because they are $G$-structures with $G$ an algebraic subgroup of $\mathcal{D}^r(\mathbb C^n)$.
\end{Remark}

For $G$ a subgroup of $\mathcal D^r(\mathbb C^n)$, a $G$-structure of order $r$ is a Gromov geometric structure of type $\mathcal{D}^r(\mathbb C^n)/G$. Indeed, if $g:\mathcal{R}^r(M)\to \mathcal{D}^r(\mathbb C^n)/G$ is a $\mathcal{D}^r(\mathbb C^n)$-equivariant map, then the subset $g^{-1}(\{id\})$ of $\mathcal{R}^r(M)$ is naturally endowed with a structure of holomorphic $G$-principal sub-bundle of $\mathcal{R}^r(M)$. 

\begin{Definition}\label{def rigidity}
    (\cite{Be97}, p.7) A Gromov geometric structure of type $Z$ and of order $r$ on $M$ is said to be rigid at order $s\geq r-1$ if for every $x\in M$ and every $t> s$, a $t$-jet of biholomorphism of $(M,x)$ that preserves the structure is completely determined by its $s$-jet.
\end{Definition}

\begin{Example}
    \begin{enumerate}
        \item (\cite{Ko95}, p.15) A jet of a germ of biholomorphism fixing a point of $M$ and preserving an absolute parallelism is in fact the jet of the identity automorphism of $M$. Therefore, an absolute parallelism is rigid at order $0$.
        \item A distribution in the tangent bundle is not rigid. As an example, we consider a germ of biholomorphism of $\mathbb C^2$ fixing $0$ and preserving the foliation generated by the vector field $\frac{\partial}{\partial z_1}$ in the standard coordinates of $\mathbb C^n$. Then, in standard coordinates, the germ of biholomorphism reads
        $$
        (z_1,z_2)\mapsto (f(z_1,z_2),g(z_2))
        $$ Conversely, any germ of this form preserves the foliation, showing that the latter is not rigid. 
        \item (\cite{Be97}, p.7) A torsion-free holomorphic affine connection is rigid at order $1$, since a jet of biholomorphism preserving a point is conjugated by the exponential chart with its differential at the fixed point.
    \end{enumerate}
\end{Example}

In his paper bringing his point of view on the demonstration of Gromov's open-dense orbit Theorem (\cite{Be97}, p.1), Benoist proved a Frobenius type theorem for existence and uniqueness of solutions of some partial derivatives relations. 

If $M$ and $N$ are complex manifolds, we denote by $J^r(M,N)$ the bundle of $r$-jets of holomorphic maps from $M$ to $N$. A partial derivatives relation is a complex submanifold $R$ of $J^r(M,N)$ and a solution is a local holomorphic map $f:M\to N$ such that the section $j^rf:M\to J^r(M,N)$ induced by $f$ takes values in $R$. 

\begin{Definition}\label{def comp et cons}
    (\cite{Be97}, p.12) $(1)$ A partial derivatives relation $R\subset J^r(M,N)$ is said to be complete if the natural projection $J^r(M,N)\to J^{r-1}(M,N)$ induces an isomorphism between $R$ and its image. \\
    $(2)$ A partial derivatives relation $R\subset J^r(M,N)$ is said to be consistent if, for every $p\in R$, there exists a section $M\to R$ through $p$ and tangent at $p$ to a holonomic section, i.e. a section of the form $j^rf$, for $f:M\to N$.
\end{Definition}

Using these two notions, Benoist proved the following.

\begin{Theorem}\label{Th Frobenius}
    (\cite{Be97}, p.12) Let $R$ be a complete and consistent partial derivatives relation. Then for every $p\in R$, there exists a unique germ of solution through $p$.
\end{Theorem}

\subsection{Holomorphic Cartan geometries}\label{Sect Cart Geo}

We will recall the concept of holomorphic Cartan geometry on a complex manifold $M$ of dimension $n$. The main references here are \cite{Sh2000} and \cite{BisDum2020}.

Let $G$ be a complex Lie group with Lie algebra $\mathfrak g$, and let $H$ be a complex Lie subgroup with Lie algebra $\mathfrak h$. 

\begin{Definition}\label{def cartan geo}
    (\cite{Sh2000}, p.184 Definition 3.1) A holomorphic Cartan geometry of type $(G,H)$ on $M$ is a holomorphic $H$-principal bundle $\pi:P\to M$ endowed with a holomorphic one form $\eta:TP\to \mathfrak g$ satisfying the following:
    \begin{enumerate}
        \item $\forall p\in P$, $\eta_p:T_pP\to \mathfrak g$ is an isomorphism (i.e. $\eta$ is an absolute parallelism on $P$) 
        \item $\forall v\in\mathfrak h$, $\eta(X_v)\equiv v$, where $X_v$ is the fundamental vector field generated by $v$ on $P$ 
        \item $\forall h\in H$, $\mathcal R_h^*\eta=Ad_{h^{-1}}(\eta)$.
    \end{enumerate}
\end{Definition}

\begin{Example}
    A holomorphic affine connection on $M$ is a holomorphic Cartan geometry of type $(Aff(n,\mathbb C),GL(n,\mathbb C)))$. Indeed, it is an absolute parallelism of the order one frame bundle $\eta:T\mathcal R^1(\mathbb C^n)\to \mathfrak {aff}(n,\mathbb C)$. The principal connection on $\mathcal R^1(M)\to M$ is given by the $Ad_{GL(n,\mathbb C)}$-invariant splitting $\mathfrak {aff}(n,\mathbb C)=T_0\mathbb C^n\oplus\mathfrak{gl}(n,\mathbb C)$. On the other hand, given a principal connection on $\mathcal R^1(M)$, we obtain the Cartan geometry by adding the data of the canonical form (see Section \ref{subsection canonical form}).
\end{Example}

We recall the central notion of curvature and a remark on this tensor field.

\begin{Definition}
    (\cite{Sh2000}, p.184) Let $\eta:TP\to \mathfrak g$ be a holomorphic Cartan geometry on $M$. The holomorphic $2$-form on $P$ defined by $\Omega:=\mathrm d\eta+\frac{1}{2}[\eta,\eta]$ is called the curvature form. The Cartan geometry $\eta$ is said to be flat if the form $\Omega$ is identically zero.  
\end{Definition}

\begin{Proposition}
    (\cite{Go2014}, p.4) Let $\eta:TP\to \mathfrak g$ be a holomorphic Cartan geometry on $M$. The curvature form $\Omega$ can be seen as a holomorphic $2$-form on $M$ with value in the $\mathfrak g$-adjoint bundle $Ad_{\mathfrak g}(P)$ of $P$ defined as the associated bundle 
    $$
    Ad_{\mathfrak g}(P):=P\times_H\mathfrak g
    $$
\end{Proposition}

The interest of the curvature is the following Theorem that asserts that a flat Cartan geometry is locally modeled on the standard Cartan geometry.

\begin{Theorem}\label{meaning of flatness for cartan geometries}
    (\cite{Sh2000}, p.212 Theorem 5.1) Let $\eta$ be a holomorphic Cartan geometry of type $(G,H)$ on a bundle $P\to M$. Then $\eta$ is flat if and only if at each $x$ of $M$, there exists a neighborhood of $x$ and a local isomorphism between the Cartan geometry $\eta$ and the bundle $G\to G/H$ equipped with the Maurer-Cartan form of $G$ (\cite{Sh2000}, p.98, definition 1.3).
\end{Theorem}

\section{Resonances of a contraction}\label{Resonances}

Let $n\geq2$ be an integer and $\gamma:\mathbb C ^n \to \mathbb C^n$ be a biholomorphic contraction in $0$, i.e. a map such that for any compact $K$ and any neighborhood $U$ of $0$, a sufficiently high power of $\gamma$ maps $K$ in $U$. Denote by $M$ the corresponding Hopf manifold, defined as the orbit space of the fixed point free and properly discontinuous action of the group $\mathbb Z$ on $\mathbb C^n\setminus \{0\}$ generated by $\gamma$, with its compact complex manifold structure given by the quotient map (\cite{Ha85}). We will prove in Section \ref{geo str ord sup} that $M$ carries geometric structures, encoded in the contraction $\gamma$ and, even more, encoded in the Jacobian of $\gamma$ at $0$.

Let $d_0\gamma:T_0\mathbb C^n\to T_0\mathbb C^n$ be the differential map of $\gamma$ at $0$. Haefliger (\cite{Ha85}, Appendix A.2) proved that the eigenvalues of this linear map are of modulus smaller than $1$. Up to a linear change of coordinates of $\mathbb C^n$, we can assume that the Jacobian matrix of $\gamma$ at 0 has, in the standard coordinates, the form 
$$
J_0\gamma=\begin{pmatrix}
\begin{matrix}
\alpha_1 &  & (0) \\
 * & \begin{matrix}
\cdot & \\ 
 & \cdot 
\end{matrix} &  \\
* &  * & \alpha_1
\end{matrix} &  & (0) \\
 & \begin{matrix}
\cdot & &  \\ 
 & \cdot &  \\
& & \cdot
\end{matrix} &  \\
(0) &  & \begin{matrix}
\alpha_m &  & (0) \\
 * & \begin{matrix}
\cdot & \\ 
 & \cdot 
\end{matrix} &  \\
* & * & \alpha_m
\end{matrix}
\end{pmatrix}
$$
where 
$$
1>\vert \alpha_1 \vert \geq ... \geq \vert \alpha_m \vert >0
$$
and the block corresponding to the eigenvalue $\alpha_j$ is of size $n_j$, for $j\in\{1,...,m\}$.

We denote also by $u_0$ the order one Taylor jet of $\gamma$ at $0$. The map $u_0:\mathbb C^n \to\mathbb C^n$ is complex linear, and its matrix in the standard basis of $\mathbb C^n$ is also given by $J_0\gamma$. In all that follows, $\alpha$ will be used to represent the data of the eigenvalues $(\alpha_1,...,\alpha_m)$ and their multiplicity $(n_1,...,n_m)$.

With this notations we will use the following central definition.

\begin{Definition}\label{def reson}
    (\cite{Bo81}, p.1)
    A $\alpha$-resonance, or simply a resonance if there is no ambiguity, is a relation of the form 
    $$
    \alpha_j=\prod\limits_{k=1}^m \alpha_k^{p_k}
    $$
    where $j$ is an element of $\{1,...,m\}$ and $p=(p_1,...,p_m)$ is a nonzero element of $\mathbb N^m$. We will say that $\alpha_i$ is dependent on $\alpha_k$ if there exists a resonance of the form 
    $$
    \alpha_j=\alpha_k\prod\limits_{l=1}^m \alpha_l^{p_l}
    $$ and independent if there exist no resonance of this type.
\end{Definition}

\begin{Remark}
    Some authors do not consider the trivial relation $\alpha_j=\alpha_j$ as a resonance. In what follows, to be as general as possible, we will need to say that $\alpha_j$ is dependent of $\alpha_j$.
\end{Remark}

\section{Geometric structures of order 1 on Hopf manifolds}\label{geo str ord 1}

We will show that every Hopf manifold carries a $G$-structure of order $1$, with $G$ a linear complex Lie group induced by the set of eigenvalues $\alpha$. Namely, let $G_\alpha^1$ be the complex Lie group of invertible triangular inferior matrices $A=(a_{i,j})_{i,j}$ such that for $i>j$, $a_{i,j}=0$ if $\alpha_{k_i}$ is independent on $\alpha_{k_j}$, where $k_i$ is the element of $\{1,...,m\}$ such that $\alpha_{k_i}$ is the diagonal coefficient of $J_0\gamma$ in the row $i$, and same for $k_j$ (see Section \ref{Resonances}).

\begin{Example}
    Assume that $n=3$ and that $\gamma$ is such that 
    $$
    J_0\gamma = \begin{pmatrix}
        \frac{1}{2} & 0 & 0  \\
        0 & \frac{1}{3} & 0  \\
        0 & 0 & \frac{1}{4}  
    \end{pmatrix}
    $$
    then $G_\alpha^1$ is the group of invertible matrices of the form 
    $$
    \begin{pmatrix}
        * & 0 & 0  \\
        0 & * & 0  \\
        * & 0 & * 
    \end{pmatrix}.
    $$

    If $\gamma$ is such that 
   $$ 
   J_0\gamma = \begin{pmatrix}
        \frac{1}{2} & 0 & 0  \\
        0 & \frac{1}{3} & 0  \\
        0 & 0 & \frac{1}{6}  
    \end{pmatrix}
  $$
   then $G_\alpha^1$ is the group of invertible matrices of the form 
    $$
    \begin{pmatrix}
        * & 0 & 0  \\
        0 & * & 0  \\
        * & * & * 
    \end{pmatrix}.
    $$

    If $\gamma$ is such that 
    $$ 
   J_0\gamma = \begin{pmatrix}
        \frac{1}{2} & 0 & 0 & 0  \\
        1 & \frac{1}{2} & 0 & 0  \\
        0 & 0 & \frac{1}{4}  & 0 \\
        0 & 0 & 0 & \frac{1}{8}
    \end{pmatrix}
  $$
  then $G_\alpha^1$ is the group of invertible matrices of the form 
   $$
    \begin{pmatrix}
        * & 0 & 0 & 0  \\
        * & * & 0 & 0  \\
        * & * & * & 0 \\
        * & * & * & *
    \end{pmatrix}.
    $$
\end{Example}

In this section, we are interested in $G^1_\alpha$-structures of order $1$ (Section \ref{G structures} for the definition of $G$-structures). To understand $G^1_\alpha$-structures of order $1$, we have to describe the group $G^1_\alpha$ as a subgroup of $GL(n,\mathbb C)$. We will realize $G^1_\alpha$ as the subgroup of $GL(n,\mathbb C)$ preserving an entanglement of sub-spaces of $\mathbb C^n$, under the standard action of $GL(n,\mathbb C)$ on $\mathbb C^n$.

For every $j\in\{1,...,m\}$, we define $I_j$ the set of non-trivial dependence indices of $\alpha_j$ i.e.
$$
I_j=\{0\}\cup \{l\in\{1,...,j-1\}~\vert~ \alpha_j \text{ depends on } \alpha_l\}.
$$
Using also the notation $V_{0,0}=\mathbb C^n$, the invertible linear maps that preserve the family of sub-spaces 
$$
V_{j,k}:=\left(\bigcap\limits_{l\in I_j} V_{l,0}\right)\cap \Vect \{e_i, i\in \{1,...,n\}\setminus\{n_1+...+n_{j-1}+1,...,n_1+...+n_{j-1}+n_j-k\}\}
$$
for every $j\in\{1,...,m\}$ and $k\in\{0,...,n_j-1\}$, are exactly the elements of $G^1_\alpha$. The vectors $e_1,...,e_n$ denote the standard basis of $\mathbb C^n$.

\begin{Example}
    Assume that $G_\alpha^1$ is the group of invertible matrices of the form 
    $$
    \begin{pmatrix}
        * & 0 & 0  \\
        0 & * & 0  \\
        * & 0 & * 
    \end{pmatrix},
    $$
    then it is the group of invertible matrices that preserve the spaces 
    $$
    \begin{matrix}
        \Vect \{e_2,e_3\}\\
        \Vect \{e_1,e_3\}\\
        \Vect\{e_2\}.
    \end{matrix}
    $$
    Assume that $G_\alpha^1$ is the group of invertible matrices of the form 
    $$
    \begin{pmatrix}
        * & 0 & 0 & 0  \\
        0 & * & 0 & 0  \\
        0 & * & * & 0  \\
        * & 0 & 0 & *
    \end{pmatrix},
    $$
    then it is the group of invertible matrices that preserve the spaces 
    $$
    \begin{matrix}
        \Vect \{e_2,e_3,e_4\}\\
        \Vect \{e_1,e_3,e_4\}\\
        \Vect\{e_1,e_4\}\\
        \Vect\{e_2,e_3\}.
    \end{matrix}
    $$
\end{Example}

Therefore, a $G_\alpha^1$-structure of order $1$ on a manifold is an entanglement of sub-bundles of the tangent bundle. The following theorem asserts that every Hopf manifold carries such a structure, descending from an invariant structure on $\mathbb C^n \setminus\{0\}$ extending to $\mathbb C^n$.

\begin{Theorem}\label{th 1}
    There exists a $G_\alpha^1$-structure on $\mathbb C^n$ invariant under the action of $<\gamma>$.
\end{Theorem}

\begin{proof}
    Since $\gamma$ is a contraction, there exists an open neighborhood $D$ of $0$ in $\mathbb C^n$ such that $\gamma(D)$ is bounded and included in $D$. Then, the space of bounded holomorphic one forms on $D$ (for a given hermitian metric) $H^0_b(D,\Omega^1_D)$ is a Banach space, and the pull-back operator 
    $$
    \gamma^* : H^0_b(D,\Omega^1_D) \longrightarrow H^0_b(D,\Omega^1_D)
    $$ is compact (\cite{Ma96}, Proposition 4.1). The compactness of this operator allows to use the powerful Riesz-Schauder theorem (\cite{F70}, Section 5.2, \cite{OV17} and \cite{OV2024} for some use of the Riesz-Schauder theorem in the context of Hopf manifolds).
    
    We will construct a $G^1_\alpha$-structure by use of an induction argument. We start with the base case.
    
    Looking at the form of the Jacobian matrix of $d_0\gamma$, we see that the $1$-form $dz_{n_1}$ is not in the image of the operator $\gamma^*-\alpha_1Id$, because there can be no linear forms $\omega_0$ on $T_0\mathbb C^n$ such that $\omega_0\circ d_0\gamma-\alpha_1\omega_0=d_0z_{n_1}$.
    Therefore, the operator $\gamma^*-\alpha_1Id$ is not onto, and since $\alpha_1$ is not equal to $0$, we infer, using the Riesz-Schauder theorem, that $\alpha_1$ is an eigenvalue of $\gamma^*$ and that there exists a positive integer $N_1$, that we can assume to be greater or equal than $n_1$, such that the space $\ker((\gamma^*-\alpha_1Id)^{N_1})$ is finite dimensional and such that the following decomposition holds
    $$
    H^0_b(D,\Omega^1_D)=\ker((\gamma^*-\alpha_1Id)^{N_1})\oplus\overline{\im((\gamma^*-\alpha_1Id)^{N_1})}.
    $$

    This allows us to consider first an arbitrary nonzero element $\omega$ of $\ker(\gamma^*-\alpha_1Id)$. We will show that $\omega$ does not vanish at $0$. To do that, we write 
    $$
    \omega=\sum_{j=1}^n\sum_{k\geq 0}P_{j,k}dz_j
    $$
    where $P_{j,k}$ is a homogeneous polynomial map of degree $k$. Since $\omega$ is nonzero, let $k_0$ be the smallest non negative integer such that one of the $P_{j,k_0}$ is nonzero, and let $j_0$ be the largest integer in $\{1,...,n\}$ such that $P_{j_0,k_0}$ is nonzero. Then the equation 
    $$
    \gamma^*\omega=\alpha_1\omega
    $$ is equivalent to 
    $$
    \sum_{j=1}^n\sum_{k\geq k_0}(P_{j,k}\circ\gamma)d\gamma_j=\alpha_1\sum_{j=1}^n\sum_{k\geq k_0}P_{j,k}dz_j
    $$
    where $\gamma=(\gamma_1,...,\gamma_n)$, which is itself equivalent to 
    $$
    \sum_{j=1}^n \left( \sum_{l=1}^n\left( \sum_{k\geq k_0}P_{l,k}\circ\gamma \right) \frac{\partial \gamma_l}{\partial z_j} \right)dz_j = \sum_{j=1}^n\left(\alpha_1 \sum_{k\geq k_0} P_{j,k}\right)dz_j.
    $$
    This last equation is equivalent to the identification of two holomorphic functions, for every $j\in \{1,...n\}$, namely
    $$
    \sum_{l=1}^n\left( \sum_{k\geq k_0}P_{l,k}\circ\gamma \right) \frac{\partial \gamma_l}{\partial z_j} = \alpha_1 \sum_{k\geq k_0} P_{j,k}.
    $$
    This implies, choosing $j=j_0$ and looking at the Taylor polynomial of order $k_0$, the following equation
    $$
    \sum_{l=1}^{j_0}(P_{l,k_0}\circ u_0) \frac{\partial \gamma_l}{\partial z_j}(0) = \alpha_1  P_{j_0,k_0}.
    $$
    But looking at the form of the Jacobian matrix $J_0\gamma$ again, we obtain that 
    $$
    \alpha P_{j_0,k_0}\circ u_0=\alpha_1 P_{j_0,k_0}
    $$
    where $\alpha$ is the eigenvalue of $J_0\gamma$ lying at row and column $j_0$, say $\alpha_r$ with $r\in \{1,...,m\}$.
    This proves that $\frac{\alpha_1}{\alpha_r}$ is an eigenvalue of the linear endomorphism 
    $$
    \begin{matrix}
        \mathbb C_{k_0}^{\Hom}[Z_1,...,Z_n] & \longrightarrow & \mathbb C_{k_0}^{\Hom}[Z_1,...,Z_n] \\
        P & \longmapsto & P \circ u_0
    \end{matrix}
    $$
    where the $\mathbb C_{k_0}^{\Hom}[Z_1,...,Z_n]$ is the space of homogeneous polynomials of degree $k_0$ with $n$ variables.
    But the eigenvalues of this endomorphism are precisely all the possible products of $k_0$ eigenvalues of $J_0\gamma$. Therefore, if $k_0$ is positive, then there is a resonance 
    $$
    \alpha_1=\alpha_r \prod_{j=1}^m \alpha_j^{p_j}
    $$
    where the length of $p=(p_1,...,p_m)$ is $k_0$, therefore positive. But this kind of relation is impossible, since 
    $$
    1>\vert \alpha_1 \vert \geq ... \geq \vert \alpha_m \vert >0.
    $$
    Hence, $k_0=0$  (and $r=1$), which proves that $\omega$ does not vanish at $0$.

    Since the space $\ker((\gamma^*-\alpha_1Id)^{N_1})$ is finite dimensional, say of dimension $r_1$, and stable by $\gamma^*$, let $(\omega_1,...,\omega_{r_1})$ be a basis of $\ker((\gamma^*-\alpha_1Id)^{N_1})$ such that $\gamma^*$ is triangular superior in this basis. We need to prove that the family $(\omega_1,...,\omega_{r_1})$ is pointwise linearly independent. Let us prove it at $0$. Assume that $\lambda_1,...,\lambda_{r_1}$ are complex numbers such that 
    $$
    \sum_{j=1}^{r_1}\lambda_j(\omega_j)_0=0.
    $$
    Then 
    $$
    \omega:=(\gamma^*-\alpha_1Id)^{N_1-1}\left(\sum_{j=1}^{r_1}\lambda_j\omega_j\right)
    $$
    is an element of $\ker(\gamma^*-\alpha_1Id)$.
 On the other hand, we have 
 $$
 \begin{matrix}
     \omega & = & \sum\limits_{j=1}^{r_1}\lambda_j(\gamma^*-\alpha_1Id)^{N_1-1}\omega_j \\
                    & = & \sum\limits_{j=1}^{r_1}\lambda_j \sum\limits_{k=0}^{N_1-1}
                    \begin{pmatrix}
                        N_1-1 \\
                        k
                    \end{pmatrix}
                    (-\alpha_1)^{N_1-1-k}(\gamma^*)^k\omega_j.
\end{matrix}
 $$
 Then 
 $$
 \begin{matrix}
     \omega_0 & = & \sum\limits_{j=1}^{r_1}\lambda_j \sum\limits_{k=0}^{N_1-1}
        \begin{pmatrix}
            N_1-1 \\
            k
        \end{pmatrix}
    (-\alpha_1)^{N_1-1-k}(\omega_j)_0\circ (d_0\gamma^k) \\
                        & = & \left( \sum\limits_{j=1}^{r_1}\lambda_j(\omega_j)_0 \right) \circ \left( \sum\limits_{k=0}^{N_1-1}
        \begin{pmatrix}
            N_1-1 \\
            k
        \end{pmatrix}
    (-\alpha_1)^{N_1-1-k}(d_0\gamma^k) \right) \\
                      & = & 0.
 \end{matrix}
 $$
 We infer that $\omega$ is identically zero. But the operator $(\gamma^*-\alpha_1Id)^{N_1-1}$ is identically zero on $\ker((\gamma^*-\alpha_1Id)^{N_1-1})$ and one-to-one onto its direct image on $\Vect(\{\omega_j \vert \omega_j\notin \ker((\gamma^*-\alpha_1Id)^{N_1-1}) \})$.  Therefore, for every $j$ such that $\omega_j\notin \ker((\gamma^*-\alpha_1Id)^{N_1-1})$, we get $\lambda_j=0$.
 Doing the same argument and using smaller and smaller powers of $\gamma^*-\alpha_1Id$, we get that $\lambda_j=0$ for every $j$.

 This proves that the family $((\omega_1)_0,...,(\omega_{r_1})_0)$ is linearly independent. By a continuity argument, there exists an open neighborhood $U$ of $0$ in $D$ such that for every $z\in U$, $((\omega_1)_z,...,(\omega_{r_1})_z)$ is linearly independent. 

 We are now proving that $r_1=n_1$ the multiplicity of $\alpha_1$ as an eigenvalue of $d_0\gamma$.
 First, notice that $((\omega_1)_0,...,(\omega_{r_1})_0)$ is a linearly independent family of elements of the characteristic space associated to $\alpha_1$ of the endomorphism
 $$
 \begin{matrix}
     (T_0\mathbb C^n)^* & \longrightarrow & (T_0\mathbb C^n)^* \\
     \phi & \longmapsto & \phi\circ (d_0\gamma)
 \end{matrix}
 $$
 and that this characteristic space has dimension $n_1$. Therefore, $r_1\leq n_1$.

 Now, for $j\in\{1,...,n_1\}$, since we have the decomposition
 $$
  H^0_b(D,\Omega^1_D)=\ker((\gamma^*-\alpha_1Id)^{N_1})\oplus\overline{\im((\gamma^*-\alpha_1Id)^{N_1})}
 $$
 there exists a unique $\eta_j\in \ker((\gamma^*-\alpha_1Id)^{N_1})$ such that $dz_j-\eta_j$ is an element of $\overline{\im((\gamma^*-\alpha_1Id)^{N_1})}$. Looking at the Jacobian $J_0\gamma$, for every $\xi\in\im((\gamma^*-\alpha_1Id)^{N_1})$, the coefficient of $\xi_0$ in front of $d_0z_l$ is zero for every $l\in\{1,...,n_1\}$. This property extends easily to the space $\overline{\im((\gamma^*-\alpha_1Id)^{N_1})}$. Moreover, $\eta_j\in \ker((\gamma^*-\alpha_1Id)^{N_1})$, so $(\eta_j)_0=d_0z_j$. Therefore, $((\eta_1)_0,...,(\eta_{n_1})_0)$ is linearly independent, so is $(\eta_1,...,\eta_{n_1})$. Since $\ker((\gamma^*-\alpha_1Id)^{N_1})$ is of dimension $r_1$, we have $n_1\leq r_1$, achieving the proof that $r_1=n_1$.

 For every $j\in\{1,...,n_1\}$ and every $z\in U$, we define 
 $$
 (H_{1,n_1-j})_z = \bigcap_{l=1}^j\ker ((\omega_l)_z).
 $$
 Then,
 $$
    \{0\}\subset H_{1,0} \subset ... \subset H_{1,n_1-1} \subset TU
    $$
    defines a $\gamma$-invariant (non complete) flag on $U$ and since $\gamma$ is a contraction in $0$ and $U$ is an open neighborhood of $0$, this flag extends to $\mathbb C^n$. We define $\mathcal H_1$ to be equal to $H_{1,0}$. Notice that the eigenvalues of $d_0\gamma$ on $(\mathcal H _1)_0$ are exactly $\alpha_2,...,\alpha_m$ with respective multiplicity $n_2,...,n_m$, since it is obvious that $d_0\gamma^*$ has a unique eigenvalue $\alpha_1$ of multiplicity $n_1$ on $(\mathcal H_1)_0^*$.

This concludes the initial case. We will then deal with the induction step.

Assume that $j\in\{1,...,m-1\}$ is such that for every $k\in\{1,...,j\}$, we have a flag of $\gamma$-invariant sub-bundles 
$$
\mathcal H_k \subset H_{k,1} \subset ... \subset H_{k,n_k-1} \subset \bigcap\limits_{l\in I_k}\mathcal H_l
$$
with 
$$
\dim \mathcal H_k=\dim \left( \bigcap_{l\in I_k}\mathcal H_l \right)-n_k
$$
and 
$$
\dim \left( \bigcap_{l\in I_k}\mathcal H_l \right)=n-\sum\limits_{l\in I_k}n_l
$$
as well as 
$$
\dim \left( \bigcap_{l\in I_{j+1}}\mathcal H_l \right)=n-\sum\limits_{l\in I_{j+1}}n_l.
$$
Finally, assume that the vector sub-bundles $\mathcal H_1,...,\mathcal H_j$ are such that for every $k\in\{1,...,j+1\}$, the eigenvalues of $d_0\gamma$ on 
$$
\left( \bigcap_{l\in I_k}\mathcal H_l \right)_0
$$
are the $\alpha_l$ such that $l\notin I_k$ and each $\alpha_l$ as multiplicity $n_l$.

The manifold 
$$
 B:=\bigcap_{k\in I_{j+1}}\mathcal H_k
$$
is a holomorphic vector bundle of rank 
$$
d:=n-\sum\limits_{l\in I_{j+1}}n_l
$$
over $\mathbb C^n$. By the Oka-Grauert principle (\cite{GR04}, p.144), it is holomorphically trivial. Let then $(s_1,...,s_{d})$ be a family of holomorphic sections trivializing $B$ and let $(s_1^*,...,s_{d}^*)$ be the dual family, trivializing the dual bundle $B^*$. Up to the action of $GL(d,\mathbb C)$, we can assume that the following holds
    $$
    Mat_{(s_1(0),...,s_{d}(0))}(d_0\gamma)=\begin{pmatrix}
\begin{matrix}
\alpha_{k_1} &  & (0) \\
 * & \begin{matrix}
\cdot & \\ 
 & \cdot 
\end{matrix} &  \\
* &  * & \alpha_{k_1}
\end{matrix} &  & (0) \\
 & \begin{matrix}
\cdot & &  \\ 
 & \cdot &  \\
& & \cdot
\end{matrix} &  \\
(0) &  & \begin{matrix}
\alpha_{k_{q}} &  & (0) \\
 * & \begin{matrix}
\cdot & \\ 
 & \cdot 
\end{matrix} &  \\
* & * & \alpha_{k_q}
\end{matrix}
\end{pmatrix}
    $$
    where 
    $$
    \{k_1<...<k_q\}=\{0,...,m\}\setminus I_{j+1}
    $$
    and the block corresponding to the eigenvalue $\alpha_{k_l}$ is of size $n_{k_l}$, for $l\in\{1,...,q\}$, as in $J_0\gamma$. Since $j+1$ is an element of $\{0,...,m\}\setminus I_{j+1}$, there exists a unique $l\in\{1,...,q\}$ such that $k_l=j+1$.

    The space $H^0_b(D,B^*)$ is a Banach space, by the Montel Theorem, and the restriction $\tilde{\gamma}$ of $\gamma^*$ to this space is itself a compact operator.

 Looking at the previous matrix, we see that 
 $$
 (s_{n_{k_1}+...+n_{k_l}})^*
 $$
 is not in the image of $\tilde{\gamma}-\alpha_{j+1}Id$. Notice that $(s_{n_{k_1}+...+n_{k_l}})^*$ is bounded since it is globally defined on $\mathbb C^n$ and $D$ is bounded, hence relatively compact. Using Riesz-Schauder Theorem, $\alpha_{j+1}$ is an eigenvalue of $\tilde{\gamma}$ and there exists a positive integer $N_{j+1}$, that we can assume to be greater or equal than $n_{j+1}$, such that the space $\ker((\tilde{\gamma}-\alpha_{j+1}Id)^{N_{j+1}})$ is finite dimensional and such that the following decomposition holds
    $$
    H^0_b(D,B^*)=\ker((\tilde{\gamma}-\alpha_{j+1}Id)^{N_{j+1}})\oplus\overline{\im((\tilde{\gamma}-\alpha_{j+1}Id)^{N_{j+1}})}.
    $$
    We take an arbitrary nonzero element $\varphi$ of $\ker(\tilde{\gamma}-\alpha_{j+1}Id)$ and we will show it does not vanish at $0$. We write
    $$
    \varphi=\sum_{\mu=1}^{d}\sum_{\nu\geq 0}P_{\mu,\nu}s_\mu^*
    $$
     where $P_{\mu,\nu}$ is a homogeneous polynomial of degree $\nu$. For every $\mu\in\{1,...,d\}$, we write 
    $$
    \tilde{\gamma}(s_\mu^*)=\sum_{\tau=1}^{d}f_{\mu,\tau}s_\tau^*
    $$
    where $f_{\mu,\tau}$ are holomorphic functions on $D$. Notice that the matrix $(f_{\mu,\tau}(0))_{\mu,\tau}$ is exactly the matrix
    $$
    \begin{pmatrix}
\begin{matrix}
\alpha_{k_1} &  & (0) \\
 * & \begin{matrix}
\cdot & \\ 
 & \cdot 
\end{matrix} &  \\
* &  * & \alpha_{k_1}
\end{matrix} &  & (0) \\
 & \begin{matrix}
\cdot & &  \\ 
 & \cdot &  \\
& & \cdot
\end{matrix} &  \\
(0) &  & \begin{matrix}
\alpha_{k_q} &  & (0) \\
 * & \begin{matrix}
\cdot & \\ 
 & \cdot 
\end{matrix} &  \\
* & * & \alpha_{k_q}
\end{matrix}
\end{pmatrix}.
    $$
    Since $\varphi$ is nonzero, let $\nu_0$ be the smallest non negative integer such that one of the $P_{\mu,\nu_0}$ is nonzero, and let $\mu_0$ be the largest integer in $\{1,...,d\}$ such that $P_{\mu_0,\nu_0}$ is nonzero. Then the equation
    $$
    \tilde{\gamma}(\varphi)=\alpha_{j+1}\varphi
    $$
    is equivalent to 
    $$
    \sum_{\mu=1}^{d}\sum_{\nu\geq \nu_0}(P_{\mu,\nu}\circ\gamma)\tilde{\gamma}(s_\mu^*)=\alpha_{j+1}\sum_{\mu=1}^{d}\sum_{\nu\geq \nu_0}P_{\mu,\nu}s_\mu^*
    $$
    which is itself equivalent to 
    $$
    \sum_{\mu=1}^{d} \left( \sum_{\tau=1}^d\left( \sum_{\nu\geq \nu_0}P_{\tau,\nu}\circ\gamma \right) f_{\tau,\mu} \right)s_\mu^* = \sum_{\mu=1}^{d}\left(\alpha_{j+1} \sum_{\nu\geq \nu_0} P_{\mu,\nu}\right)s_\mu^*.
    $$
    This last equation is equivalent to the identification of two holomorphic maps, for every $\mu\in \{1,...,d\}$, namely
    $$
    \sum_{\tau=1}^{d}\left( \sum_{\nu\geq \nu_0}P_{\tau,\nu}\circ\gamma \right) f_{\tau,\mu} = \alpha_{j+1} \sum_{\nu\geq \nu_0} P_{\mu,\nu}.
    $$
    This implies, choosing $\mu=\mu_0$, and looking at the Taylor polynomial of order $\nu_0$, the following equation
    $$
    \sum_{\tau=1}^{\mu_0}(P_{\tau,\nu_0}\circ u_0) f_{\tau,\mu_0}(0) = \alpha_{j+1}  P_{\mu_0,\nu_0}.
    $$
    But looking at the form of the matrix $(f_{\mu,\tau}(0))_{\mu,\tau}$, we obtain that 
    $$
    \alpha P_{\mu_0,\nu_0}\circ u_0=\alpha_{j+1} P_{\mu_0,\nu_0}
    $$
    where $\alpha$ is the eigenvalue of $(f_{\mu,\tau}(0))_{\mu,\tau}$ lying at row and column $\mu_0$, say $\alpha_r$ with $r\in \{0,...,m\}\setminus I_{j+1}$, which exactly means that $r=j+1$ or $\alpha_{j+1}$ does not depend on $\alpha_r$.
    This proves that $\frac{\alpha_{j+1}}{\alpha_r}$ is an eigenvalue of the linear endomorphism 
    $$
    \begin{matrix}
        \mathbb C_{\nu_0}^{\Hom}[Z_1,...,Z_n] & \longrightarrow & \mathbb C_{\nu_0}^{\Hom}[Z_1,...,Z_n] \\
        P & \longmapsto & P \circ u_0.
    \end{matrix}
    $$
    But the eigenvalues of this endomorphism are precisely all the possible products of $\nu_0$ eigenvalues of $d_0\gamma$. Therefore, if $\nu_0$ is positive, then there is a resonance 
    $$
    \alpha_{j+1}=\alpha_r \prod_{i=1}^m \alpha_i^{p_i}
    $$
    where the length of $p=(p_1,...,p_m)$ is $\nu_0$, therefore positive. But this kind of relation is impossible, since $r\in \{0,...,m\}\setminus I_{j+1}$.
    Hence, $\nu_0=0$  (and $r=j+1$), which proves that $\varphi$ does not vanish at $0$. 
    
    Since $\ker((\tilde{\gamma}-\alpha_{j+1}Id)^{N_{j+1}})$ is finite dimensional, say of dimension $r_{j+1}$, and stable by $\tilde{\gamma}$, let $(\varphi_1,...,\varphi_{r_{j+1}})$ be a basis of $\ker((\tilde{\gamma}-\alpha_{j+1}Id)^{N_{j+1}})$ such that $\tilde{\gamma}$ is triangular superior in this basis. Then, as in the initial case, we can prove that $((\varphi_1)_0,...,(\varphi_{r_{j+1}})_0)$ is linearly independent, and projecting the family 
    $$
    ((s_{n_{k_1}+...+n_{k_{l-1}}+1})^*,...,(s_{n_{k_1}+...+n_{k_{l}}})^*)
    $$
    on $\ker((\tilde{\gamma}-\alpha_{j+1}Id)^{N_{j+1}})$, as in the initial case, that $r_{j+1}=n_{j+1}$.
    We can then consider, for $k\in\{1,...,n_{j+1}\}$
    $$
    H_{j+1,n_{j+1}-k}=\bigcap\limits_{l=1}^{k}\ker (\varphi_l)
    $$
    in a neighborhood of $0$ and extend those $\gamma$-invariant vector bundles to $\mathbb C^n$, applying the contraction $\gamma$ as much as we need. We denote by $\mathcal H_{j+1}$ the vector bundle $H_{j+1,0}$.
    By construction, we have a flag
    $$
    \mathcal H_{j+1} \subset H_{j+1,1} \subset ... \subset H_{j+1,n_{j+1}-1} \subset \bigcap\limits_{k\in I_{j+1}}\mathcal H_k
    $$
    with 
    $$
    \dim \mathcal H_{j+1}=\dim \bigcap\limits_{k\in I_{j+1}}\mathcal H_k -n_{j+1}.
    $$
    To conclude the induction step, we need to show that, if $j<m-1$, then 
    $$
    \dim \bigcap\limits_{k\in I_{j+2}}\mathcal H_k=n-\sum_{k\in I_{j+2}}n_k.
    $$
    We check first this dimension in $T_0\mathbb C^n$. The spectrum of $d_0\gamma$ on 
    $$
    \left(
    \bigcap \limits _{k\in I_{j+2}}\mathcal H_k
    \right)_0
    $$
    is exactly $\{\alpha_l,~l\notin I_{j+2}\}$ and each $\alpha_l$ has multiplicity $n_l$. Therefore, at $0$, the dimension is what we need. Let $z$ be an element of $\mathbb C^n$ and let $d$ be the dimension of 
    $$
    \left(
    \bigcap \limits _{k\in I_{j+2}}\mathcal H_k
    \right)_z.
    $$
    Since the sequence $(\gamma^p(z))_{p\in\mathbb N}$ converges to $0$, and since $T_0\mathbb C^n$ is trivial, the sequence of vector sub-spaces of $\mathbb C^n$ 
    $$
    \left(d_{\gamma^{p-1}(z)}\gamma\left(
    \bigcap \limits _{k\in I_{j+2}}\mathcal H_k
    \right)_z\right)_{p\in\mathbb N}
    $$
    converges to
    $$
    \left(
    \bigcap \limits _{k\in I_{j+2}}\mathcal H_k
    \right)_0
    $$
    so we have that $d=n-\sum_{k\in I_{j+2}}n_k$.
    Therefore, 
    $$
    \bigcap \limits _{k\in I_{j+2}}\mathcal H_k
    $$
    is a vector bundle of rank $n-\sum_{k\in I_{j+2}}n_k$.
    Hence we proved the induction step, and by induction, we proved the theorem.
\end{proof}

This theorem implies that every Hopf manifold carries special geometries, namely a $G$-structure of order $1$, with a group that is only determined by the eigenvalues of the differential map at $0$ of the contraction. In the following section, we will define a group of polynomial transformations of $\mathbb C^n$. The group $G^1_\alpha$ can be identified as the group of jets of order $1$ of elements of this group.

\section{The group of sub-resonant polynomial transformations}\label{group}

From now on, we will use the following notation
$$
J_0\gamma=
\begin{pmatrix}
    \beta_1 &   &  (0) \\
     & 
     \begin{matrix}
         \cdot & & \\
          & \cdot & \\
           & & \cdot
     \end{matrix}
     & \\
    * & & \beta_n
\end{pmatrix}.
$$
so that 
$$
\begin{matrix}
    \beta_1=...=\beta_{n_1}=\alpha_1 \\
    \cdot \\
    \cdot \\
    \cdot \\
    \beta_{n_1+...+n_{m-1}+1}=...=\beta_{n}=\alpha_m.
\end{matrix}
$$

We define a positive integer $r$ depending on the set of eigenvalues $\{\beta_1,...,\beta_n\}$, namely the maximal length of the resonances, as follows
$$
r=\max\{\vert p \vert ~, p\in \mathbb N^n~\vert ~ \exists ~j\in \{1,...,n\}~ \vert~ \beta_j=\beta^p\}.
$$
Notice that the fact that $r$ is finite was proven by Poincaré (\cite{A88}, Theorem 1 p.188). 

We will define a group of polynomial transformations of $\mathbb C^n$ of order $r$. This group depends on the set of resonances of $\beta=(\beta_1,...,\beta_n)$. For $j\in\{1,...,n\}$ we consider
$$
\begin{matrix}
    r_j=\max \{\vert p \vert ~, p\in \mathbb N^n~\vert ~ \beta_j=\beta^p\}\\
    \\
    \mathcal P_j=\{p\in \mathbb N^n~\vert~ \beta_j=\beta^p \text{ and }~\forall k\geq j, p\neq e_k\}\\
    \\
    \mathcal Q_j=\{q\in \mathbb N^n\setminus \{0\}~\vert~\exists p\in \mathcal P_j~\vert~\forall l\in\{1,...,n\}, q_l\leq p_l\}
\end{matrix}
$$
where $e_k$ is the $k$-th element of the standard basis of $\mathbb Z^n$. The number $r_j$ is the maximal length of resonance of $\beta_j$ (and therefore we have $r=\max r_j$), the set $\mathcal{P}_j$ is the set of resonances of $\beta_j$, where we do not take into account $\beta_j=\beta_{j+l}$ with $l\geq 0$, even if it is the case, and $\mathcal Q_j$ is the set of nonzero $n$-tuples of natural integers that are component-wise not greater than some resonance of $\beta_j$. We refer to Example \ref{exemple poly sub} for clarifications on those technical definitions.

Let $P:\mathbb C^n\to \mathbb C^n$. We say that $P$ is a $\beta$ sub-resonant polynomial transformation of $\mathbb C^n$ if, in the standard coordinates, 
$$
P(z)=(a_1z_1,a_2z_2+P_2(z_1),...,a_nz_n+P_n(z_1,...,z_{n-1}))
$$
where 
$$
\begin{matrix}
    \forall j\in\{1,...,n\}, & a_j\neq 0 \\
    \forall j\in\{2,...,n\}, & P_j(z_1,...,z_{j-1})=\sum\limits_{q\in \mathcal Q_j}c_{j,q}z^q.
\end{matrix}
$$
Notice that such a polynomial map is at most of degree $r$.

\begin{Example}\label{exemple poly sub}
    Assume that $\gamma$ is such that 
    $$
    J_0\gamma=
    \begin{pmatrix}
        \alpha & 0 & 0 \\
        0 & \alpha^3 & 0 \\
        0 & 0 & \alpha^5
    \end{pmatrix}
    =
    \begin{pmatrix}
        \beta_1 & 0 & 0 \\
        0 & \beta_2 & 0 \\
        0 & 0 & \beta_3
    \end{pmatrix}.
    $$
    The non trivial resonances are
    $$
    \begin{matrix}
        \beta_2 & = & \beta_1^3 & = & \beta^{(3,0,0)} \\
        \beta_3 & = & \beta_1^5 & = & \beta^{(5,0,0)} \\
        \beta_3 & = & \beta_1^2\beta_2 & = & \beta^{(2,1,0)}
    \end{matrix}
    $$
    Therefore, we have 
    $$
    \begin{matrix}
        r=5 & r_1=1 & r_2=3 & r_3=5
    \end{matrix}
    $$
    
    $$
    \begin{matrix}
          \mathcal P_1=\emptyset & \mathcal Q_1=\emptyset \\
          \mathcal P_2=\{(3,0,0)\} & \mathcal Q_2=\{(1,0,0),(2,0,0),(3,0,0)\} \\
          \mathcal P_3=\{(5,0,0),(2,1,0)\} & \mathcal Q_3=\{(i,0,0), 1\leq i \leq 5\}\cup\{(0,1,0),(1,1,0),(2,1,0)\} 
    \end{matrix}
    $$

    The polynomial map $P$ given by the expression
    $$
    P(z)=(z_1,z_2+2z_1-z_1^3,z_3+z_1+z_1z_2-2z_1^2z_2 +z_1^4+z_1^5)
    $$
    is sub-resonant. Indeed, the monomials $z_1^3e_2$, $z_1^2z_2e_3$, $z_1^5e_3$ are resonant, i.e. they correspond to elements of the sets $\mathcal P_j$, and the monomials $z_1e_2$, $z_1e_3$, $z_1z_2e_3$, $z_1^4e_3$ are sub-resonant i.e. they correspond to elements of the sets $\mathcal Q_j$.
\end{Example}
We denote by $G^r_\beta$ the set of all $\beta$ sub-resonant polynomial transformations of $\mathbb C^n$. This is a set of polynomial maps of degree at most $r$. We will show that it is a subgroup of the group of holomorphic automorphisms of $\mathbb C^n$.  

\begin{Proposition}\label{proof group}
    The set $G^r_\beta$, endowed with the composition of maps from $\mathbb C^n$ to itself, is a subgroup of the group $Aut(\mathbb C^n)$ of all the holomorphic automorphisms of $\mathbb C^n$.
\end{Proposition}
\begin{proof}
    The following proof is technical, but we will give some intuitive ideas first. It is already known that the set of $\beta$ resonant polynomial maps is a group for the composition of maps (\cite{LN96}, Section 1 Proposition 2). The idea is that since the sub-resonant monomials are by definition not greater than some resonant monomial and since the composition of resonant monomials remains resonant, the composition of sub-resonant monomials remains sub-resonant.
    
    We will show this result by induction on the dimension $n$.
    If $n=2$, then let $P,\tilde{P}$ be elements of $G^r_\beta$. We use the same notations as in the definition above, namely
    $$
    \begin{matrix}
        P(z)=(a_1z_1,a_2z_2+P_2(z_1)) \\
        \tilde{P}(z)=(b_1z_1,b_2z_2+\tilde{P_2}(z_1)).
    \end{matrix}
    $$
    A monomial appearing in $\tilde{P_2}$ is a $z^q$ with $q\in\mathcal{Q}_2$, i.e. a $z_1^{q_1}$ with $1\leq q_1\leq d$ where $\beta_2=\beta_1^d$. In case we don't have this type of resonances, then $\mathcal Q_2$ is empty and $G^r_\beta$ is obviously a group, since it is the group of linear diagonal transformations of $\mathbb C^2$.
    In the resonant case, we have that 
    $$
    z_1^{q_1}\circ P (w)=a_1^{q_1}w_1^{q_1}.
    $$
    This identity proves that $\tilde{P}\circ P$ is an element of $G^r_\beta$.
    Moreover, $P$ is obviously one-to-one, and 
    $$
    P^{-1}(w)=\left(\frac{w_1}{a_1},\frac{w_2-P_2(\frac{w_1}{a_1})}{a_2}\right)
    $$
    which is also an element of $G^r_\beta$. Therefore, if $n=2$, then $G^r_\beta$ is a subgroup of $Aut(\mathbb C^2)$.

    Let now $n\in \mathbb N$ be such that for every integer $m$ satisfying $2\leq m \leq n$ and every set of eigenvalues $(\eta_1,...,\eta_m)$ satisfying 
    $$
    1\geq \vert \eta_1 \vert \geq ... \geq \vert \eta_m \vert \geq 0
    $$
    the set $G_\eta$ is a group.
    Let $(\beta_1,...,\beta_{n+1})\in \mathbb C^{n+1}$ satisfying 
    $$
     1\geq \vert \beta_1 \vert \geq ... \geq \vert \beta_{n+1} \vert \geq 0.
    $$
    Let $P,\tilde{P}$ be $\beta$ sub-resonant. Looking at the composition $\tilde{P}\circ P$, the $n$ first coordinates form a composition of two $(\beta_1,...,\beta_n)$ sub-resonant polynomial maps. Therefore, using the induction hypothesis, we only focus on the $(n+1)$-th coordinate.
    The polynomial $\tilde{P}_{n+1}(z_1,...,z_n)$ contains monomials of the form $z^q$ with $q\in\mathcal Q_{n+1}$. By definition of $\mathcal Q_{n+1}$, there exists an element $p$ of $\mathcal P_{n+1}$ such that $q\leq p$ (coordinate by coordinate).
    We use the notation 
    $$
    P(z)=(a_1z_1+P_1(z),...,a_{n+1}z_{n+1}+P_{n+1}(z))
    $$
    where $a_j\neq 0$ and 
    $$
    P_j(z)=\sum_{s\in\mathcal Q_j}c_{j,s}z^s.
    $$
    Then we compute 
    $$
   \begin{matrix}
       z^q\circ P(w) & = & \prod\limits_{j=1}^{n+1}(a_jw_j+P_j(w))^{q_j} \\
                                     & = & \prod\limits_{j=1}^{n+1}(a_jw_j+\sum\limits_{s\in\mathcal Q_j}c_{j,s}w^s)^{q_j}.
   \end{matrix} 
    $$
    We denote by $N_j$ the cardinality of $\mathcal Q_j$ and $\mathcal Q_j=\{s_1^{(j)},...,s_{N_j}^{(j)}\}$. Then we have
    $$
    \begin{matrix}
       z^q\circ P(w) & = & \prod\limits_{j=1}^{n+1} \left(
            \sum\limits_{i_1+...+i_{N_j+1}=q_j}\frac{q_j!}{i_1!...i_{N_j+1}!} \left(
\prod\limits_{k=1}^{N_j}\left(
\left(
c_{j,s_k^{(j)}}w^{s_k^{(j)}}
\right)^{i_k}
\right)
            \right)(a_jw_j)^{i_{N_j+1}}
       \right) \\
       & = & 
       \prod\limits_{j=1}^{n+1} \left(
\sum\limits_{i_1+...+i_{N_j+1}=q_j}\frac{q_j!}{i_1!...i_{N_j+1}!} \left(
\prod\limits_{k=1}^{N_j}c_{j,s_k^{(j)}}^{i_k}
\right)a_j^{i_{N_j+1}}w^{
\left(
\sum\limits_{k=1}^{N_j}i_ks_k^{(j)}
\right)+i_{N_j+1}e_j
}
       \right)
   \end{matrix}
    $$
    where, as usual, $(e_1,...,e_{n+1})$ is the standard basis of $\mathbb Z^{n+1}$, so that the monomials appearing are $w$ to the power of  
    $$
    \sum_{j=1}^{n+1}\left(
    \sum_{k=1}^{N_j}i_k^{(j)}s_k^{(j)}+i_{N_j+1}^{(j)}e_j
    \right)
    $$
    where, for every $j\in\{1,...,n+1\}$, we have
    $$
    \sum_{k=1}^{N_j+1}i_k^{(j)}=q_j.
    $$
    We need to explain that for every 
    $$
    (i_{1}^{(1)},...,i_{N_1+1}^{(1)}),...,(i_{1}^{(n+1)},...,i_{N_{n+1}+1}^{(n+1)})
    $$
    satisfying 
    $$
    \sum_{k=1}^{N_j+1}i_k^{(j)}=q_j
    $$ the $(n+1)$ tuple 
    $$
    \sum_{j=1}^{n+1}\left(
    \sum_{k=1}^{N_j}i_k^{(j)}s_k^{(j)}+i_{N_j+1}^{(j)}e_j
    \right)
    $$
    is bounded above by an element of $\mathcal P_{n+1}$.
    First, we notice that $q_{n+1}=0$, implying that, for every $k\in\{1,...,N_{n+1}+1\}$, we have $i_k^{(n+1)}=0$. We can rewrite the $(n+1)$ tuple as 
    $$
    \sum_{j=1}^{n}\left(
    \sum_{k=1}^{N_j}i_k^{(j)}s_k^{(j)}+i_{N_j+1}^{(j)}e_j
    \right).
    $$
    By definition of the sets $\mathcal Q_j$, for every $j\in\{1,...,n\}$, for every $k\in\{1,...,N_j\}$, there exists an element $\underline{p}_k^{(j)}$ of $\mathcal P_j$ such that $s_k^{(j)}\leq \underline{p}_k^{(j)}$. We recall that $p$ is an element of $\mathcal P_{n+1}$ such that $q\leq p$. We consider 
    $$
    \tilde{p}=p-\sum_{j=1}^n\left(
\sum_{k=1}^{N_j}i_k^{(j)}
    \right)e_j+\sum_{j=1}^n\sum_{k=1}^{N_j}i_k^{(j)}\underline{p}_k^{(j)}
    $$
    and we check that $\beta^{\tilde{p}}=\beta_{n+1}$, meaning that $\tilde{p}$ is an element of $\mathcal P_{n+1}$. Since $p\geq q$ and $\underline{p}_k^{(j)}\geq s_k^{(j)}$, we have 
    $$
    \tilde{p}\geq q -\sum_{j=1}^n\left(
\sum_{k=1}^{N_j}i_k^{(j)}
    \right)e_j+\sum_{j=1}^n\sum_{k=1}^{N_j}i_k^{(j)}s_k^{(j)},
    $$
    but we have 
    $$
    q -\sum_{j=1}^n\left(
\sum_{k=1}^{N_j}i_k^{(j)}
    \right)e_j=\sum_{j=1}^ni_{N_j+1}^{(j)}e_j.
    $$
    Therefore, we deduce that 
    $$
    \tilde{p} \geq \sum_{j=1}^{n}\left(
    \sum_{k=1}^{N_j}i_k^{(j)}s_k^{(j)}+i_{N_j+1}^{(j)}e_j
    \right)
    $$
    which is what we wanted. This shows that $\tilde{P}\circ P$ is an element of $G_\beta$.
    To conclude, we write 
    $$
    P^{-1}(w)=\left(
\frac{w_1}{a_1},\frac{w_2-P_2(\frac{w_1}{a_1})}{a_2},...,\frac{w_{n+1}-P_{n+1}(...)}{a_{n+1}}
    \right).
    $$
    The induction hypothesis allows us to look only at the $(n+1)$-th coordinate, that we treat exactly the same way as for the composition.

    This concludes the induction and therefore, the proof that $G^r_\beta$ is a subgroup of $Aut(\mathbb C^n)$.   
\end{proof}

To conclude this section, we notice that the group $G^r_\beta$ can be seen as a subgroup of the jets group $\mathcal D^r(\mathbb C^n)$, and that it is a closed subgroup. Therefore, it carries a natural structure of complex Lie subgroup of $\mathcal D^r(\mathbb C^n)$.

\begin{Remark}
    In relation to Remark \ref{rmk algebraic}, we notice also that $G^r_\beta$ is an algebraic subgroup of $\mathcal{D}^r(\mathbb C^n)$.
\end{Remark}

\section{Geometric structures of higher order on Hopf manifolds}\label{geo str ord sup}

For every $k\in \mathbb N^*$, we consider the group 
$$
G^k_\beta:=\{j^k_0 P ~\vert~ P\in G^r_\beta\}
$$
endowed with the order $k$ truncated composition. 
Notice that since $G^r_\beta$ is a group of polynomial maps of degree at most $r$, for every $k\geq r$, we have $G^k_\beta=G^r_\beta$. Notice also that $G^1_\beta$ coincides with the $G_\alpha^1$ of Section \ref{geo str ord 1}. We will show that we can extend a given $G^1_\beta$-structure of order $1$ step by step to get a uniquely determined $G_\beta^k$-structure of order $k$, for any positive integer $k$ in the case of dimension $n\geq3$, and for $k\in\{r+1,r+2\}$ for surfaces.

The case of Hopf surfaces has to be treated separately. The reason for this special treatment is a cohomological property that is different between $\mathbb C^2\setminus\{0\}$ and $\mathbb C^n\setminus\{0\}$, for $n\geq 3$. Indeed, if $\mathcal{O}_X$ stands for the the sheaf of holomorphic functions on a complex manifold $X$, then we have the following properties (see for example \cite{Fr57} or \cite{LN96}, Section 2.1)
$$
\left\{
\begin{matrix}
    H^1(\mathbb C^n\setminus\{0\},\mathcal{O}_{\mathbb C^n\setminus\{0\}})=0 \text{ for } n\geq 3 \\
    H^1(\mathbb C^2\setminus\{0\},\mathcal{O}_{\mathbb C^2\setminus\{0\}})\cong \left\{\sum\limits_{p\in (-\mathbb N^*)^2}a_pz^p \text{ convergent on }(\mathbb C^*)^2\right\}.
\end{matrix}
\right.
$$
Because of this phenomenon, Mall's theorem fails in dimension two, as noticed in \cite{Ma96}, Remark 4.4.. In Section \ref{cas surfaces}, we will deal with the case of Hopf surfaces.

\subsection{High dimension}\label{cas sup deux}

This section is devoted to the proof of the following result. 

\begin{Theorem}\label{recurrence pour les structures}
    Assume that $n\geq3$. Let $k$ be a positive integer and let $s^{(k)}:M\to \mathcal R^k(M)/G^k_\beta$ be a reduction of the structure group of $\mathcal R^k(M)$ to $G^k_\beta$. Then there exists a unique reduction $s^{(k+1)}:M\to \mathcal R^{k+1}(M)/G^{k+1}_\beta$ of the structure group of $\mathcal R^{k+1}(M)$ to $G^{k+1}_\beta$ such that
    $$
    \overline{\pi}_{k+1,k}\circ s^{(k+1)}=s^{(k)}
    $$
    where 
    $$
    \overline{\pi}_{k+1,k}:\mathcal R^{k+1}(M)/G^{k+1}_\beta \to \mathcal R^k(M)/G^k_\beta
    $$
    is the natural projection.
\end{Theorem}

\begin{proof}
    Let $M=\cup_i U_i$ be a covering of $M$ by open subsets and let $(\varphi_i)$ local holomorphic charts on $U_i$ such that 
    $$
    s^{(k)}\vert_{U_i}=[s^{(k)}_i]
    $$
    where 
    $$
    \begin{matrix}
         s^{(k)}_i & : & M & \longrightarrow & \mathcal R^k(M) \\
                               &    & m & \longmapsto & j^k_0(\varphi_i^{-1}\theta_{\varphi_i(m)}f_{i,m})
    \end{matrix}
    $$
    where $\theta_z$ is the translation in $\mathbb C^n$ of vector $z\in\mathbb C^n$ and $j^k_0(f_{i,m})$ is the jet of order $k$ at $0$ of a germ of biholomorphism $(\mathbb C^n,0)\to (\mathbb C^n,0)$.
    We consider the following diagram
    \begin{center}
            \begin{tikzcd}
               (s^{(k)})^*\left( \mathcal R^{k+1}(M)/G^{k+1}_\beta \right) \arrow[r] \arrow[d,"p"] & \mathcal R^{k+1}(M)/G^{k+1}_\beta \arrow[d,"\overline{\pi}_{k+1,k}"] \\
                M \arrow[r,"s^{(k)}"] & \mathcal R^k(M)/G^k_\beta
            \end{tikzcd}
    \end{center}
    where the upper left object is the pullback of the diagram. We will denote by $A$ the space 
    $$
    (s^{(k)})^*\left( \mathcal R^{k+1}(M)/G^{k+1}_\beta \right).
    $$

    We claim that the map $p:A\to M$ is an affine bundle. The fiber of this bundle is isomorphic to the quotient 
    $$
    V^{k+1}/V^{k+1}_\beta
    $$
    where $V^{k+1}$ is the space of homogeneous polynomial maps of degree $k+1$ from $\mathbb C^n$ to $\mathbb C^n$ and $V^{k+1}_\beta$ is the subspace generated by the monomials of the form 
    $$
    z^qe_j
    $$ where $q\in \mathcal Q_j$ is of length $k+1$, i.e. $q$ is a sub-resonance of $\beta_j$ of length $k+1$. To see that, we consider 
    $$
    \begin{matrix}
        \psi_i & : & A\vert_{U_i} & \longrightarrow & U_i\times V^{k+1}/V^{k+1}_\beta \\
                       &    & (m,[j_0^{k+1}(\varphi_i^{-1}\theta_{\varphi_i(m)}(f_{i,m})_k(id+P))]) & \longmapsto & (m,[P])
    \end{matrix}
    $$
    where $(f_{i,m})_k$ is the truncating at order $k$ of $f_{i,m}$.
    We compute 
    $$
    \begin{matrix}
        \psi_j\psi_i^{-1}(m,[P]) & = & \psi_j(m,[j^{k+1}_0(\varphi_i^{-1}\theta_{\varphi_i(m)}(f_{i,m})_k(id+P))]) \\ 
        &&\\
         & = & \psi_j(m,[j^{k+1}_0(\varphi_j^{-1}\theta_{\varphi_j(m)}(f_{j,m})_k(((f_{j,m})_k)^{-1}\theta_{-\varphi_j(m)}\varphi_j\\& &\varphi_i^{-1}\theta_{\varphi_i(m)}(f_{i,m})_k)(id+P))]). 
    \end{matrix}
    $$
    But, by construction,
    $$
    j^k_0(((f_{j,m})_k)^{-1}\theta_{-\varphi_j(m)}\varphi_j\varphi_i^{-1}\theta_{\varphi_i(m)}(f_{i,m})_k)
    $$
    is an element of $G^k_\beta$. Therefore, we see that 
    $$
    \psi_j\psi_i^{-1}(m,[P])=(m,[\mathcal L_{j,i,m}(P)+Q_{j,i,m}])
    $$
    where 
    $$
    \mathcal L_{j,i,m}(P)=\phi_{j,i,m} \circ P\circ \phi_{j,i,m}^{-1}
    $$
    with 
    $$
    \phi=((f_{j,m})_1)^{-1}(\theta_{-\varphi_j(m)}\varphi_j\varphi_i^{-1}\theta_{\varphi_i(m)})_1(f_{i,m})_1
    $$
    and $Q_{j,i,m}\in V^{k+1}$ is independent of $P$.
    This proves that $p$ is an affine bundle with typical fiber $V^{k+1}/V^{k+1}_\beta$. 

    To prove the theorem is equivalent to prove that the affine bundle $p$ admits a unique global section. We will use a cohomological argument in order to prove that. Namely, if $\xi,\zeta$ are two local sections of $p$, say on an open set $U$, then the difference $\xi-\zeta$ is well defined as a section of the vector bundle, denoted by $B$, associated with $p$. Since $p$ is locally trivial, we can choose local sections on trivialization open sets and consider the difference of those sections on the intersection of the open sets of the covering. Therefore the existence of a global section of $p$ is equivalent to the vanishing of a cohomology class in the first sheaf cohomology group of the associated vector bundle $B$. Moreover, a global section of the affine bundle is unique if and only if there exists a unique global section of $B$, namely the zero section. To rephrase, the sub-resonant monomials are exactly the indeterminacy to extend the structure to the superior order. Extending up to these monomials, i.e. considering the quotient $ V^{k+1}/V^{k+1}_\beta$, determines completely how to extend the structure. 

    Let us then study this holomorphic vector bundle $B$ on $M$. Using the above computation, we have a description of the cocycle of $B$. Namely, it is given by $\phi_{i,j,m}$. Therefore, we can see that the principal bundle associated with $B$ is the first order frame bundle on $M$, i.e. $\mathcal R^1(M)$. Indeed, they have the same cocycle, up to the action of the cochain $\{(f_{i,m})_1\}$. This shows that if $\pi:\mathbb C^n\setminus \{0\} \to M$ is the universal covering of $M$, then the pulled back vector bundle $\overline{B}:=\pi^*B$ is trivial over $\mathbb C^n\setminus \{0\}$ and it extends to $\mathbb C^n$. Therefore, $B$ is a Mall bundle over the Hopf manifold $M$. Since we are in the case $n\geq3$, the Mall theorem implies that 
    $$
    h^0(M,B)=h^1(M,B)
    $$
    where $h^i(M,B)$ is the dimension of the $i$-th sheaf cohomology space of the vector bundle $B$. Using that, it is sufficient to prove that $H^0(M,B)$ is trivial to finish the proof of the theorem.

    The fiber over $0$ of the bundle $\overline{B}$, namely $\overline{B}_0$ is isomorphic to the space 
    $$
    V^{k+1}/V^{k+1}_\beta
    $$
    and in a well chosen isomorphism, the action of $\gamma$ on $\overline{B}_0$ reads
    $$
    \gamma.P=u_0^{-1}\circ P \circ u_0
    $$
    where, as above, $u_0$ is the linear part of $\gamma$ and where $P$ is a homogeneous polynomial map of order $k+1$ from $\mathbb C^n$ to itself containing only non sub-resonant monomials, because $P$ is an element of $V^{k+1}/V^{k+1}_\beta$. We deduce that the eigenvalues of the action of $\gamma$ on $\overline{B}_0$ are exactly of the form 
    $$
    \beta_j^{-1}\beta^p
    $$
    where the length of $p$ is $k+1$ and $p$ is not $j$ sub-resonant, i.e. $p\notin \mathcal Q_j$ where this last set is defined above.

    We choose $(\eta_1,...,\eta_d)$ a family of global trivializing sections of $\overline{B}$. Suppose that $\overline{B}$ carries a nonzero global section $\eta$. Then we can write 
    $$
    \eta=\sum_{i=1}^d\sum_{\nu\geq 0}P_{i,\nu}\eta_i
    $$
    where $P_{i,\nu}$ is a homogeneous polynomial of degree $\nu$. Since $\eta$ is nonzero, let $\nu_0$ the minimal degree where not all the $P_{i,\nu}$ are zero. Suppose also that $\eta$ is $\gamma$ invariant, then looking at the equation 
    $$
    \gamma.\eta=\eta
    $$
    and looking at the order $\nu_0$, we deduce that the linear endomorphism
    $$
    \begin{matrix}
        \mathbb C^{\Hom} _{\nu_0}[Z_1,...,Z_n]\otimes V^{k+1}/V^{k+1}_\beta & \longrightarrow & \mathbb C^{\Hom}_{\nu_0}[Z_1,...,Z_n]\otimes V^{k+1}/V^{k+1}_\beta \\
        Q\otimes P & \longmapsto & (Q\circ u_0) \otimes (u_0^{-1}\circ P \circ u_0)
    \end{matrix}
    $$
    carries $1$ as an eigenvalue. But this implies that there exists a resonance 
    $$
    \beta_j=\beta^p\beta^{p'}
    $$
     where the length of $p$ is $k+1$ and $p$ is not $j$ sub-resonant, and where the length of $p'$ is $\nu_0$. But this is impossible, because it would imply that $p$ is sub-resonant. 

     This proves that 
     $$
     H^0(M,B)=\{0\}
     $$
     which, as we said above, is enough to conclude the proof of the theorem.
    
\end{proof}

This theorem combined with the existence of the $G^1_\beta$-structure of order $1$ proved in Section \ref{geo str ord 1} allow to write the following corollary, proved by induction.

\begin{Corollary}\label{existence de structure a tout ordre}
    Let $M$ be a Hopf manifold of dimension at least $3$ and $\beta=(\beta_1,...,\beta_n)$ be the set of eigenvalues associated with this Hopf manifold. Then, for every integer $k\geq 1$, $M$ carries $G^k_\beta$-structure of order $k$. Moreover, every $G^k_\beta$-structure of order $k$ is completely determined by the induced $G^1_\beta$-structure of order $1$.
\end{Corollary}

\subsection{Hopf surfaces}\label{cas surfaces}

In this subsection, we focus on Hopf surfaces. What fails in the proof of Theorem \ref{recurrence pour les structures}, is the moment when we use Mall's Theorem giving the equality
$$
h^0(M,B)=h^1(M,B).
$$
Notice that the proof of the fact that $h^0(M,B)=0$ is still valid. Therefore, if we can extend a structure of order $k$ at the order $k+1$, this extension is unique. But the problem is that, \textit{a priori}, the space $H^1(M,B)$ is not trivial.

In their paper \cite{OV2024}, Ornea and Verbitsky have the same problem. As they mention, Theorem 6.6 is still true. We will explain why.

\subsubsection{Mall theorem for Hopf surfaces}

In \cite{Ma96}, Mall shows a theorem that allows to compute the cohomology of any holomorphic vector bundle over a Hopf manifold of dimension $n\geq3$ with a trivial pullback on $\mathbb C^n\setminus\{0\}$. Let us state his Theorem.

\begin{Theorem}\label{Mall theorem}
    (\cite{Ma96}, Theorem 1.1) Let $B$ be a Mall bundle on $M$ a Hopf manifold of dimension $n\geq3$. Then we have 
    $$
   \left\{
   \begin{matrix}
       h^0(M,B)&=&h^1(M,B) & \\
       h^p(M,B)&=&0 & \text{for}~2\leq p \leq n-2 \\
       h^{n-1}(M,B)&=&h^n(M,B) & 
   \end{matrix}
   \right.
    $$
\end{Theorem}
We want to give a version of this theorem for Hopf surfaces. As mentioned by Mall, by using the argument in the proof of Theorem \ref{Mall theorem}, we only see that the Euler characteristic vanishes. The same result can be achieved by adding a hypothesis for Hopf surfaces. 

Let $M=W/<\gamma>$ be a Hopf surface, where $W:=\mathbb C^2\setminus \{0\}$ and $\gamma:\mathbb C^2\to \mathbb C^2$ is a biholomorphic contraction in $0$. We denote by $\pi:W\to M$ the projection. 

Let $B$ be a Mall bundle of rank $r$ on $M$. Since $\pi^*(B)$ is trivial, there exists a group representation 
$$
\begin{matrix}
    \rho_B & : & <\gamma> & \longrightarrow & GL(r,\Gamma(W,\mathcal O))\\
           &   & \gamma & \longmapsto & A 
\end{matrix}
$$
such that the bundle $B$ is isomorphic to the quotient $E_A$ of the trivial bundle $W\times \mathbb C^r$ by the action 
$$
\begin{matrix}
    <\gamma> \times (W\times \mathbb C^r) & \longrightarrow & (W\times \mathbb C^r) \\
    (\gamma,(z,v)) & \longmapsto & (\gamma(z), A(z)v).
\end{matrix}
$$

Mall uses the generalized Douady sequence (see \cite{Dou60}, p.5). For Hopf surfaces, this exact sequence reads 

$$
\begin{tikzcd}
    0 \arrow[r] & H^0(M,E_A) \arrow[r,"\pi^*"] & H^0(W,\mathcal O^r) \arrow[r,"A-\gamma^*"] & H^0(W,\mathcal O^r) \arrow[r] & H^1(M,E_A) \\ \arrow[r,"\pi^*"] & H^1(W,\mathcal O^r) \arrow[r,"A-\gamma^*"] & H^1(W,\mathcal O^r) \arrow[r] & H^2(M,E_A) \arrow[r] & 0.
\end{tikzcd}
$$
From this exact sequence and with the work of Mall we infer the following result.

\begin{Theorem}
\label{Mall theorem for surfaces}
    Let $E_A$ be a Mall bundle on $M$. 
    
    If the map 
    $$
    \begin{matrix}
        A-\gamma^* & : & H^1(W,\mathcal O^r) & \longrightarrow & H^1(W,\mathcal O^r)
    \end{matrix}
    $$
    is one-to-one, then we have $h^0(M,E_A)=h^1(M,E_A)$.
    
    If this map is onto, then we have $H^2(M,E_A)=0$.
\end{Theorem}
\begin{proof}
    The second assertion is straightforward. The first assertion comes from the Corollary 4.2 in \cite{Ma96}  ensuring that 
    $$
    \begin{matrix}
        A-\gamma^* & : & H^0(W,\mathcal O^r) & \longrightarrow & H^0(W,\mathcal O^r)
    \end{matrix}
    $$
    is a Fredholm operator of index $0$, even for surfaces.
\end{proof}

\subsubsection{Ornea and Verbitsky theorem for Hopf surfaces}

In \cite{OV2024}, Ornea and Verbitsky proved their Theorem 6.6 for Hopf manifolds of dimension at least $3$ and mentioned that the result was true for surfaces. We give here the proof for Hopf surfaces because we judge it interesting in order to understand what will come next in the resonant case. 

Let $\gamma$ be a non-resonant contraction of $\mathbb C^2$ at $0$. In this section exceptionally, non-resonant means that the maximal length of resonance is $1$. To be more specific, the linear map $d_0\gamma$ can have one eigenvalue of multiplicity $2$. Here, this is the definition of ``non-resonant" that Ornea and Verbitsky use in \cite{OV2024}. Let $M$ be the corresponding Hopf manifold. 

\begin{Theorem}\label{OV surfaces}
    (\cite{OV2024}, p.26, Theorem 6.6) The Hopf surface $M$ carries a unique flat torsion-free holomorphic connection. Moreover, there exists a biholomorphism $\varphi:\mathbb C^2\to\mathbb C^2$
such that $\varphi\gamma\varphi^{-1}$ is linear.
\end{Theorem}

\begin{proof}
    The proof is the same as the one of Ornea and Verbitsky, except that we have to show that $h^0(M,TM^*\otimes TM^*\otimes TM)=h^1(M,TM^*\otimes TM^*\otimes TM)$. To do that, according to Theorem \ref{Mall theorem for surfaces}, we need to show that the map 
    $$
    \begin{matrix}
        H^1(\mathbb C^2\setminus \{0\},(\mathcal O^2)^*\otimes(\mathcal O^2)^*\otimes\mathcal O^{2}) & \longrightarrow & H^1(\mathbb C^2\setminus \{0\},(\mathcal O^2)^*\otimes(\mathcal O^2)^*\otimes\mathcal O^{2})
    \end{matrix}
    $$
    taking a convergent Laurent series on $(\mathbb C^*)^2$
    $$
    s(z)=\sum\limits_{j,k=1}^2\sum\limits_{l=1}^2\sum\limits_{p\in (-\mathbb N^*)^2}a_{j,l,m,p}z^pdz_j\otimes dz_k \otimes \frac{\partial}{\partial z_l}
    $$
    to 
    $$
    \sum\limits_{j,k=1}^2\sum\limits_{l=1}^2\sum\limits_{p\in (-\mathbb N^*)^2}a_{j,l,m,p}z^pdz_j\circ (d_z\gamma)^{-1}\otimes dz_k \circ (d_z\gamma)^{-1}  \otimes d_z\gamma_*\frac{\partial}{\partial z_l}-s(\gamma(z))
    $$
    is one-to-one. 

    Let $s$ be such a convergent power series satisfying that its image is $0$. Then assuming that $s$ is nonzero, and looking at the lowest nonzero term in the Taylor expansion, we get a relation of the form 
    $$
    \alpha_1^j\alpha_2^{2-j}=\alpha_l\alpha_1^{-p_1}\alpha_2^{-p_2}
    $$
    where $\alpha_1,\alpha_2$ are the eigenvalues of $d_0\gamma$, $j$ is an element in $\{0,1,2\}$, $l$ is an element in $\{1,2\}$ and $p=(p_1,p_2)$ is an element of $(-\mathbb N^*)^2$. This relation automatically gives rise to a resonance of length greater than $1$, which contradicts our assumptions. Therefore, the following equality holds
    $$
    h^0(M,TM^*\otimes TM^*\otimes TM)=h^1(M,TM^*\otimes TM^*\otimes TM).
    $$
    The rest of the proof is exactly the same as in \cite{OV2024}.
\end{proof}

\subsubsection{G-structures on resonant Hopf surfaces}\label{geo str on surf}

We could try to prove Theorem \ref{recurrence pour les structures} the same way for Hopf surfaces. Nevertheless, once we reach the computation of 
$$
H^1(M,B)
$$
where $B$ has typical fiber isomorphic to $V^{k+1}/V^{k+1}_\beta$, the computation of the map 
$$
H^1(\mathbb C^2\setminus\{0\},\pi^*B)\longrightarrow H^1(\mathbb C^2\setminus\{0\},\pi^*B)
$$
assuming that it is not one-to-one produces relations of the form
$$
\alpha_1^j\alpha_2^{k+1-j}=\alpha_l\alpha_1^{-p_1}\alpha_2^{-p_2}
$$
and, as soon as $k\geq2$, the existence of this type of relation is not a contradiction.
Therefore, the space $H^1(M,B)$ is not trivial, which means that the existence of an extension of our geometric structure is a class in $H^1(M,B)$ which is, \textit{a priori}, non trivial.

Let $\gamma:\mathbb C^2\to \mathbb C^2$ be a holomorphic contraction at $0$, and let $M$ be the corresponding Hopf manifold. Let $\alpha_1,\alpha_2$ be the eigenvalues of $d_0\gamma$, verifying 
$$
1>\vert \alpha_1  \vert \geq \vert \alpha_2 \vert >0.
$$
According to Theorem \ref{th 1}, there exists a $<\gamma>$-invariant $G^1_\alpha$-structure of order $1$ on $\mathbb C^2$. A consequence of the existence of this structure is that there exists a holomorphic distribution of lines in $T\mathbb C^2$ invariant under the action of $<\gamma>$.
\begin{Lemma}\label{forme globale et fermee}
    The $<\gamma>$-invariant distribution is given by the kernel of a global holomorphic nowhere vanishing $1$-form $\omega$ on $\mathbb C^2$ which verifies the following properties:
    \begin{enumerate}
        \item the form $\omega$ is closed,
        \item the following equation holds: $\gamma^*\omega=\alpha_1\omega$.
    \end{enumerate}
\end{Lemma}

\begin{proof}
    To prove this lemma, we need to go back to the proof of Theorem \ref{th 1}. According to that proof, there exist a bounded open neighborhood $D$ of $0\in\mathbb C^2$ verifying that $\gamma(D)$ is relatively compact in $D$, and a non-vanishing holomorphic one form $\omega_0$ on $D$ satisfying the equation $\gamma^*\omega_0=\alpha_1\omega_0$.  The distribution is given locally by $\ker(\omega_0)$, and it extends to $\mathbb C^2$ because $\gamma$ is a contraction.

    Now, using again the fact that $\gamma$ is a contraction, we have the following equality:
    $$
    \mathbb C^2=\bigcup\limits_{n\in\mathbb N} \gamma^{-n}(D).
    $$
    For every $n\in\mathbb N$, we define a holomorphic $1$-form on $\gamma^{-n}(D)$ by the following formula:
    $$
    (\omega_n)_z=\frac{1}{\alpha_1^n}(\omega_0)_{\gamma^n(z)}\circ d_z\gamma^n.
    $$
    Then the family of nowhere vanishing  holomorphic one forms $(\omega_n)_{n\in\mathbb N}$ verifies that 
    $$
    \left\{
    \begin{matrix}
        \omega_{n+1}\vert_{\gamma^{-n}(D)}=\omega_n \\
        \gamma^*\omega_n=\alpha_1\omega_{n+1}
    \end{matrix}
    \right.
    $$
    Therefore, all those forms patch together to produce a global nowhere vanishing holomorphic one form $\omega$ on $\mathbb C^2$ verifying $\gamma^*\omega=\alpha_1\omega$.

    To prove that $\omega$ is closed, we consider its exterior derivative $\mathrm d\omega$. This is a holomorphic two form on $\mathbb C^2$ verifying $\gamma^*\mathrm d\omega=\alpha_1\mathrm d\omega$. We assume that $\mathrm d\omega$ is nonzero and we consider its power series expansion. Then, the equation $\gamma^*\mathrm d\omega=\alpha_1\mathrm d\omega$ implies, looking at the lowest nonzero order of the Taylor expansion, a relation of the type
    $$
    \alpha_1=\alpha_j\alpha_k\alpha^p
    $$
    where $j,k$ are elements of $\{1,2\}$ and $p$ is an element of $\mathbb N^2$. But the existence of such a relation is a contradiction because of the property
    $$
    1>\vert \alpha_1  \vert \geq \vert \alpha_2 \vert >0.
    $$
    Therefore, the form $\mathrm d\omega$ is the zero form which implies that $\omega$ is closed.
\end{proof}

Now, there are essentially three cases in our situation. 

The first one is the non-resonant case (see Definition \ref{def reson}). In this situation, the proofs of Theorem \ref{th 1} and Lemma \ref{forme globale et fermee} show that there exists another global holomorphic closed one form $\eta$ on $\mathbb C^2$ satisfying $\gamma^*\eta=\alpha_2\eta$. Moreover, $\omega$ and $\eta$ are pointwise linearly independent. Since $\mathbb C^2$ is simply connected, we consider the holomorphic functions $w_1$ and $w_2$ on $\mathbb C^2$, vanishing at $0$, and such that $\omega=\mathrm dw_1$ and $\eta=\mathrm dw_2$. Then $\varphi=(w_1,w_2)$ is a local biholomorphism, and it verifies that 
$$
\varphi \gamma \varphi^{-1}=
\begin{pmatrix}
    \alpha_1 & 0 \\
    0 & \alpha_2
\end{pmatrix}.
$$
In other words, $\omega$ and $\eta$ define a flat torsion-free affine connection  on $\mathbb C^2$ invariant by $<\gamma>$, and $\varphi$ is the developing map of this affine connection. This gives a new proof of Theorem 6.6 in \cite{OV2024}, p.27, in our non-resonant case.

The second one is the case where $\alpha_1=\alpha_2$. In this case, according to the proof of Theorem \ref{th 1}, there exists a nowhere vanishing holomorphic one form $\eta_0$ on $D$ satisfying $\gamma^*\eta_0=\alpha_1\eta_0+\varepsilon\omega_0$, where $\varepsilon$ is $0$ or $1$, and such that $\omega_0$ and $\eta_0$ are pointwise linearly independent. 

\begin{Lemma}\label{eta closed}
    The form $\eta_0$ is closed.
\end{Lemma}

\begin{proof}
    Since we have the equality $\gamma^*\eta_0=\alpha_1\eta_0+\varepsilon\omega_0$, and since $\omega_0$ is closed by Lemma \ref{forme globale et fermee} we infer that the exterior differential of $\eta_0$ satisfies the equation $\gamma^*\mathrm d\eta_0=\alpha_1\mathrm d\eta_0$. But then, by the same Taylor expansion argument as the one in the proof of Lemma \ref{forme globale et fermee}, the $2$-form $\mathrm d\eta_0$ is identically $0$. Therefore, $\eta_0$ is closed.
\end{proof}

As in the non resonant case, we integrate $\omega$ and $\eta_0$ in a neighborhood of $0\in\mathbb C^2$ and we produce local coordinates $\varphi$ such that 
$$
\varphi \gamma \varphi^{-1}=
\begin{pmatrix}
    \alpha_1 & 0 \\
    \varepsilon & \alpha_1
\end{pmatrix}.
$$
This completes the new proof of Theorem 6.6 in \cite{OV2024}, p.27 in the non resonant case.

Let us now deal with the third case, which is the remaining resonant case, namely, where there exists an integer $r\geq2$ such that $\alpha_2=\alpha_1^r$. In this case, we will denote by $\lambda$ the eigenvalue $\alpha_1$, so that $\alpha_2=\lambda^r$. 

\begin{Lemma}\label{vector field}
    The distribution $\ker(\omega)$ is also given by a non-vanishing global holomorphic vector field $X$ on $\mathbb C^2$ satisfying the equation $\gamma_*X=\lambda^rX$.
\end{Lemma}

\begin{proof}
    According to the proof of Theorem \ref{th 1}, there exists a section $\eta_0$ of the bundle $\ker(\omega)^*$ over $D$ satisfying the equation $(\eta_0)_{\gamma(z)}\circ d_z\gamma\vert_{\ker(\omega)_z}=\lambda^r(\eta_0)_z$. Using the same argument as in the proof of the fact that $\omega_0$ extends to $\mathbb C^2$ (see Lemma \ref{forme globale et fermee}), $\eta_0$ extends to a global non-vanishing section of the bundle $\ker(\omega)$ verifying $\eta_{\gamma(z)}\circ d_z\gamma\vert_{\ker(\omega)_z}=\lambda^r\eta_z$. Now, since the bundle $\ker(\omega)$ is a line bundle (here we take advantage of the dimension $2$), $\eta$ produces a trivialization of the bundle $\ker(\omega)$. Namely, define $X$ as the section of $\ker(\omega)$ satisfying $\eta(X)=1$. Then, $X$ can be seen as a non-vanishing holomorphic vector field on $\mathbb C^2$ generating the distribution of lines $\ker(\omega)$. We check that the equation $\eta_{\gamma(z)}\circ d_z\gamma\vert_{\ker(\omega)_z}=\lambda^r\eta_z$ implies that $\gamma_*X=\lambda^rX$, which completes the proof.
\end{proof}

At this moment, we need to prove that the germ of the contraction at $0$ can be transformed, under a local change of coordinates, into a simpler form. This simpler form is not the Poincaré-Dulac normal form that we will see down below (Section \ref{Poincare Dulac}). But this simpler form will help us to see geometric structures of higher order on our Hopf surface.

\begin{Lemma}\label{integration forme et vec f}
    There exists a local system of coordinates $\varphi:U\to V$ between two open neighborhoods of $0\in\mathbb C^2$ such that $\varphi\gamma\varphi^{-1}$ has the form 
    $$
    \begin{matrix}
        (w_1,w_2) & \longmapsto & (\lambda w_1,\lambda^r w_2+f(w_1))
    \end{matrix}
    $$
    where $f$ is a local holomorphic function of one variable, verifying $f(0)=0$.
\end{Lemma}

\begin{proof}
    According to Lemma \ref{vector field}, there exists a nowhere vanishing vector field $X$ on $\mathbb C^2$ satisfying $\gamma_*X=\lambda^rX$.
    The flow-box theorem asserts that there exists a system of local coordinates $(\xi_1,\xi_2)$ in a neighborhood of $0$ such that $X=\frac{\partial}{\partial\xi_2}$. In these coordinates, the holomorphic $1$-form $\omega$ reads 
    $$
    \omega=g(\xi_1,\xi_2)d\xi_1+h(\xi_1,\xi_2)d\xi_2.
    $$
    The condition that $X$ belongs to the kernel of $\omega$ (see Lemma \ref{vector field}) implies that the map $h$ is zero. Now, we use that the form $\omega$ is closed (see Lemma \ref{forme globale et fermee}). This implies that the map $g$ does not depend on $\xi_2$. We therefore rewrite 
    $$
    \omega=g(\xi_1)d\xi_1.
    $$
    Let $G$ be the anti-derivative of $g$ vanishing at $0$. Then we have $\omega=\mathrm d G$. 
    We consider the map 
    $$
    (\xi_1,\xi_2)\mapsto (w_1,w_2)=(G(\xi_1),\xi_2).
    $$
    Since $\omega$ does not vanish, $G$ is a local biholomorphism, and the latter map is a change of coordinates in a neighborhood of $0$. In this new system of coordinates, the vector field $X$ reads $\frac{\partial}{\partial w_2}$ and the form $\omega$ reads $\mathrm dw_1$.
    Now, the equations $\gamma^*\omega=\lambda\omega$ and $\gamma_*X=\lambda^rX$ imply that, in the coordinates $(w_1,w_2)$, the contraction $\gamma$ reads 
    $$
    \begin{matrix}
        (w_1,w_2) & \longmapsto & (\lambda w_1,\lambda^r w_2+f(w_1))
    \end{matrix}
    $$
    which proves the lemma.
\end{proof}

The consequence of this lemma is that we immediately see that the contraction $\gamma$ preserves (integrable) $G$-structures of arbitrary order, namely, if we denote by $\beta$ the couple of eigenvalues $(\lambda,\lambda^r)$, then $\gamma$ obviously preserves a $G^k_\beta$-structure of order $k$ for every $1\leq k\leq r$. To see that, it suffices to notice that for any element $(x_1,x_2)$ in $\mathbb C^2$, the map 
$$
(w_1,w_2)\mapsto (\lambda(w_1+x_1),\lambda^r(w_2+x_2)+f(w_1+x_1))-(\lambda x_1,\lambda^r x_2+f(x_1))
$$
has its $k$-jet at $0$ in the group $G^k_\beta$. In other words, at each point where the map 
$$
    \begin{matrix}
        (w_1,w_2) & \longmapsto & (\lambda w_1,\lambda^r w_2+f(w_1))
    \end{matrix}
$$
is defined, the translation of this map to build a map fixing the origin has its $k$-jet in the group $G^k_\beta$. Notice that all those structures are compatible. Indeed, they are all induced by the one of order $r$. We denote by $R_r$ the principal $G^r_\beta$ sub-bundle of $\mathcal R^r(\mathbb C^2)$ given by the structure of order $r$.

What is natural to do is to try to extend the $G^r_\beta$-structure of order $r$ in a $G^r_\beta$-structure of order $r+1$, using the technique of Theorem \ref{recurrence pour les structures}.  The problem here is that the vector bundle we consider in the proof still have a non-trivial first cohomology set.

Let us introduce new groups. We denote by $H_\beta^{r+1}$ (respectively $H^{r+2}_\beta$) the group $G^{r+1}_{\beta'}$ (respectively $G^{r+1}_{\beta'}$), where $\beta'$ denotes the set of eigenvalues $(\lambda,\lambda^{r+2})$.  Notice that the introduction of this new set of eigenvalues is purely artificial. The reason to do that is to use what we already did to prove that $H_\beta^{r+1}$ and $H^{r+2}_\beta$ are complex Lie groups (see Section \ref{group}). Using these definitions, Lemma \ref{integration forme et vec f} shows that $\gamma$ preserves a $H_\beta^{r+1}$-structure of order $r+1$, that will be denoted by $\widetilde{R}_{r+1}$, and a $H_\beta^{r+2}$-structure of order $r+2$, that will be denoted by $\widetilde{R}_{r+2}$. Those structures are again compatible with the others, i.e. all the structures we consider are induced by $\widetilde{R}_{r+2}$.

We consider now extensions of $R_r$ at order $r+2$ as sub-bundles of $\widetilde{R}_{r+2}$ instead of sub-bundles of $\mathcal{R}^{r+2}(\mathbb C^2)$. This has the effect of changing the vector bundle that we consider and now we can argue as in Theorem \ref{recurrence pour les structures} to prove the following.

\begin{Theorem}\label{extension surfaces}
    There exists a unique holomorphic section of the bundle $\widetilde{R}_{r+2}\to R_r$ which is $G^r_\beta$-equivariant and $<\gamma>$-equivariant.
\end{Theorem}

\begin{proof}
    We start with the problem of finding a $G^r_\beta$-equivariant and $<\gamma>$-equivariant section of $\widetilde{R}_{r+1}\to R_r$. The latter is an affine line bundle, with fiber isomorphic to $\mathbb CZ_1^{r+1}e_2$ where $Z_1^{r+1}e_2$ is considered as the homogeneous polynomial map $\mathbb C^2\to\mathbb C^2$ of order $r+1$ given by $(z_1,z_2)\mapsto(0,z_1^{r+1})$. We infer that, by Hartogs' theorem, it is enough to work on $\mathbb C^2\setminus\{0\}$. To deal with the $<\gamma>$-equivariance, we work on $M$. We recall that all the geometric structures over $\mathbb C^2$ that we consider are $<\gamma>$-invariant. 

    The difference between two local equivariant sections of $\widetilde{R}_{r+1}\to R_r$ is a local section of the induced vector bundle $B$ over $M$ with fiber $\mathbb CZ_1^{r+1}e_2$. 
    
    As usual (see proofs of Theorems \ref{recurrence pour les structures} and \ref{OV surfaces}), we need to show, using Mall theorem on surfaces (Theorem \ref{Mall theorem for surfaces}), that $h^0(M,B)=h^1(M,B)=0$.

    We start by proving that $h^0(M,B)=0$. We take a global section of $B$, we pull it back to $\mathbb C^2\setminus\{0\}$ and extend it to $\mathbb C^2$ using Hartogs Theorem. Considering the Taylor expansion at $0$ of the section and writing the $\gamma$-invariance, we infer that, if the section is nonzero, there exists a relation of the type 
    $$
    \alpha_2=\alpha_1^{r+1}\alpha^p
    $$
    where $p=(p_1,p_2)$ is an element of $\mathbb N^2$. In our case, this relation reads
    $$
    \lambda^r=\lambda^{r+1+p_1+rp_2}
    $$
    and leads to a contradiction since $0<\vert\lambda\vert<1$.

    Therefore, $h^0(M,B)=0$. This implies that an eventual global equivariant section of $\widetilde{R}_{d+1}\to R_d$ would be unique, and that, using Theorem \ref{Mall theorem for surfaces} and the same computations as in the proof of Theorem \ref{OV surfaces}, we will have $h^1(M,B)=0$ if a relation of the following type leads to a contradiction:
    $$
    \alpha_1^{r+1}=\alpha_2\alpha^{-p}
    $$ where $p=(p_1,p_2)$ is an element of $(-\mathbb N^*)^2$. However, such a relation reads
    $$
    \lambda^{r+1}=\lambda^{r-p_1-rp_2}
    $$
    which is a contradiction. Therefore, we have that $h^1(M,B)=0$.

    We infer that there exists a unique $G^r_\beta$-equivariant section of $R_r\to\widetilde{R}_{r+1}$ over the manifold $M$. Pulling back everything on $\mathbb C^2\setminus\{0\}$ and extending everything to $\mathbb C^2$ by Hartogs' Theorem, we get a unique $G^r_\beta$-equivariant section of $R_r\to\widetilde{R}_{r+1}$ over the manifold $\mathbb C^2$ which is $<\gamma>$-equivariant. We denote by $R_{r+1}$ the image of this section and this is a $<\gamma>$-invariant $G^r_\beta$-structure of order $r+1$ over $\mathbb C^2$.

    We apply the same argument to find an equivariant section of the bundle $\widetilde{R}_{r+2}\to R_{r+1}$. The only thing that is different is that we have to consider relations of type 
    $$
    \lambda^{r+2}=\lambda^{r-p_1-rp_2}
    $$
    where $p=(p_1,p_2)$ is an element of $(-\mathbb N^*)^2$. Such a relation cannot exist because, by assumption, $r\geq 2$.

    Therefore, there exists a unique equivariant section of $\widetilde{R}_{r+2}\to R_{r+1}$, and composing with the first section we computed, we get the result.
\end{proof}

Reading the proof of Theorem \ref{extension surfaces}, we can give two remarks. The first one is that the proof breaks if $r=1$, i.e. if $d_0\gamma$ has a double eigenvalue. Fortunately, we discussed this case above. The second one is that we cannot produce with this proof for $G^r_\beta$-structures of any order, namely, relations of the form 
$$
\lambda^{2r+1}=\lambda ^{r-p_1-rp_2},
$$ where $p=(p_1,p_2)$ is an element of $(-\mathbb N^*)^2$, can exist. Indeed, it suffices to choose $p_1=p_2=-1$. Since $r\geq2$, Theorem \ref{extension surfaces} shows that we can \textit{a priori} extend only twice the structure of order $r$. However, we will see in the next section that this is sufficient.

\section{Cartan connection, integrability and Poincaré- Dulac theorem}\label{Poincare Dulac}

\subsection{Another Lie group}\label{new lie group}

In this subsection, we introduce a new complex Lie group of interest. This Lie group will also depend on the data $\beta$ of the eigenvalues of the contraction generating the Hopf manifold. We begin with a central lemma.

\begin{Lemma}\label{conj trans G beta}
    For any element $P$ in the complex Lie group $G^r_\beta$ and for any $z$ in $\mathbb C^n$, the automorphism of $\mathbb C^n$ given by $\theta_{-P(z)}P\theta_z$, where $\theta_w$ is the translation of vector $w$, is an element of $G^r_\beta$.
\end{Lemma}

\begin{proof}
    Since $P$ is a polynomial map of order $r$, the automorphism $\theta_{-P(z)}P\theta_z$, sending $0$ to $0$, is also a polynomial map of order $r$. Moreover, in the Taylor expansion at $0$ of the latter polynomial map, the monomials appearing are ``bounded above" by those of $P$, in the sense that if $Z^pe_j$ appears in the expansion of $\theta_{-P(z)}P\theta_z$, where $Z=(Z_1,...,Z_n)$, $Z^p=Z_1^{p_1}...Z_n^{p_n}$ and $e_j$ is as usual the $j$-th standard coordinate of $\mathbb C^n$, then there exists $q\in \mathbb N^n$ such that $p\leq q$ component by component, and such that $Z^qe_j$ appears in $P$. But then $q$ is $j$ sub resonant for the eigenvalues $\beta$, in the sense of Section \ref{group}, and therefore, $p$ is also $j$ sub resonant for the eigenvalues $\beta$.
    Therefore, the polynomial map $\theta_{-P(z)}P\theta_z$ is an element of $G^r_\beta$.
\end{proof}

We consider the subgroup $\widetilde{G^r_\beta}$ of the group of automorphisms of $\mathbb C^n$ generated by $G^r_\beta$ and the translations. Then, as a corollary of Lemma \ref{conj trans G beta}, we get the following.

\begin{Corollary}\label{biholo with the product}
    The group $\widetilde{G^r_\beta}$ is a complex Lie group, biholomorphic to the complex manifold $\mathbb C^n\times G^r_\beta$. Through this biholomorphism, the group operation reads
    $$
    (z,P)(w,Q)=(z+P(w),\theta_{-P(w)}P\theta_wQ).
    $$
\end{Corollary}

\begin{proof}
    By definition, we have 
    $$
    \widetilde{G^r_\beta}=<G^r_\beta,\{\theta_z, z\in  \mathbb C^n\}>_{Aut(\mathbb C^n)}.
    $$
    We consider the map 
    $$
    \begin{matrix}
        \rho & : & \mathbb C^n\times G^r_\beta & \longrightarrow & \widetilde{G^r_\beta} \\
        &&(z,P) & \longmapsto & \theta_zP.
    \end{matrix}
    $$
    This map is one-to-one, because if $\theta_zP=\theta_wQ$, then $w$ and $z$ are the respective images of $0$ by both automorphisms, so $z=w$ and then $P=Q$. Now the map $\rho$ is onto because of Lemma \ref{conj trans G beta}. Indeed, an element $f$ of $\widetilde{G^r_\beta}$ is a finite composition of elements of $G^r_\beta$ and translations. As an automorphism of $\mathbb C^n$, it sends the origin to a vector $z=f(0)$. Now the composition $P=\theta_{-z}f$ is an element of $\widetilde{G^r_\beta}$ fixing the origin. But by Lemma \ref{conj trans G beta}, such an element belongs to $G^r_\beta$. Therefore, we have $f=\rho(z,P)$.

    Computing the composition $\theta_zP\theta_wQ$, we get (notice that it illustrates the previous claims)
    $$
    \begin{matrix}
        \theta_zP\theta_wQ & = & \theta_z\theta_{P(w)}\theta_{-P(w)}P\theta_wQ \\
        & = & \theta_{z+P(w)}\theta_{-P(w)}P\theta_wQ \\
        & = & \rho(z+P(w),\theta_{-P(w)}P\theta_wQ).
    \end{matrix}
    $$

    To finish the proof, we notice that the complex manifold $\mathbb C^n\times G^r_\beta$ endowed with the group operation given by $\rho$ is naturally a complex Lie group, since every operation involved is holomorphic.
\end{proof}

By definition, the complex Lie group $G^r_\beta$ is a Lie subgroup of $\widetilde{G^r_\beta}$. In the next subsection, we will see that the geometric structures that we build in Section \ref{geo str ord sup} induce a Cartan geometry on the Hopf manifold with model $(\widetilde{G^r_\beta},G^r_\beta)$ (see Section \ref{Sect Cart Geo}). 

We recall from Section \ref{subsection canonical form} the existence of the representation $\underline{Ad}^{r+1}:\mathcal{D}^{r+1}(\mathbb C^n)\to GL(T_0\mathbb C^n\times\mathfrak d^{r}(\mathbb C^n))$. Using Corollary \ref{biholo with the product}, the Lie algebra of $\widetilde{G^r_\beta}$ is isomorphic (as complex vector spaces, not as Lie algebras) with $T_0\mathbb C^n\times \mathfrak g^r_\beta$, where $\mathfrak g^r_\beta$ is the Lie algebra of $G^r_\beta$. Therefore, it can be seen as a subspace of $T_0\mathbb C^n\times\mathfrak d^{r}(\mathbb C^n)$. 

\begin{Lemma}\label{adjoint est bien adjoint}
    Let $P$ be an element of $G^r_\beta$. The space $T_0\mathbb C^n\times \mathfrak g^r_\beta$ is preserved by the linear map $\underline{Ad}_P^{r+1}$, where $P$ is seen as an element of $\mathcal D^r(\mathbb C^n)$. Moreover, when restricted to $T_0\mathbb C^n\times \mathfrak g^r_\beta$, the linear transformation $\underline{Ad}_P^{r+1}$ coincides with the adjoint representation of the element $P$ in the group $\widetilde{G^r_\beta}$.
\end{Lemma}

\begin{proof}
    By definition of $\underline{Ad}^{r+1}$ (see Section \ref{subsection canonical form}), we have that 
    $$
    \begin{matrix}
        \underline{Ad}_P^{r+1} & = & \mathrm d_{j^r_0id}\xi_r\circ\mathrm d_{j^r_0id}(j^r_0f\mapsto j^r_0(PfP^{-1}))\circ(\mathrm d_{j^r_0id}\xi_r)^{-1} \\
        &=& \mathrm{d}_{(0,j^r_0id)}(\xi_r\circ j^r_0f\mapsto j^r_0(PfP^{-1}) \circ \xi_r^{-1}) \\
        &=& \mathrm{d}_{(0,j^r_0id)}((z,j^r_0\phi)\mapsto (P(z),j^r_0(\theta_{-P(z)}P\theta_z\phi P^{-1}))).
    \end{matrix}
    $$
    But if $z$ is a vector in $\mathbb C^n$ and $\phi$ is an element of $G^r_\beta$, then by Proposition \ref{proof group}
 and by Lemma \ref{conj trans G beta}, then $(P(z),j^r_0(\theta_{-P(z)}P\theta_z\phi P^{-1}))$ is an element of $\mathbb C^n\times G^r_\beta$. Therefore, the space $T_0\mathbb C^n\times \mathfrak g^r_\beta$ is preserved by the linear map $\underline{Ad}_P^{r+1}$. 
 
 Moreover, computing the interior homomorphism $Int_P$ in the group $\widetilde{G^r_\beta}$, we have
 $$
 P\theta_wQP^{-1}=\theta_{P(w)}\theta_{-P(w)}P\theta_wQP^{-1},
 $$
 so by Corollary \ref{biholo with the product}, we have the equality
 $$
 \underline{Ad}_P^{r+1}\vert_{T_0\mathbb C^n\times \mathfrak g^r_\beta}=Ad_P
 $$ where the adjoint representation of the right hand side is the one of the group $\widetilde{G^r_\beta}$.
 
 \end{proof}

\subsection{The Cartan geometry induced by the G-structures}\label{cartan geo}

We recall that $r$ is the maximal length of the resonances of the set of eigenvalues $\beta=(\beta_1,...,\beta_n)$. Using this definition, we deduce that for every $k\geq r$, we have 
$$
G^k_\beta=G^r_\beta.
$$
Moreover, since $G^r_\beta$ is a Lie group for the usual composition (not the truncated composition), for every $k\geq r$, $G^r_\beta$ can be seen as a Lie subgroup of $\mathcal{D}^k(\mathbb C^n)$, the group of jets of order $k$ at $0$ of germs of biholomorphism $(\mathbb C^n,0)\to (\mathbb C^n,0)$. 

If the dimension $n$ is not lower than $3$, then, Corollary \ref{existence de structure a tout ordre} shows that there exists a $G^r_\beta$-structure of order $r$ on $M$. If $n=2$, this is proved in Section \ref{geo str on surf}. This structure can be written as a $G^r_\beta$-principal sub-bundle $R_r$ of $\mathcal R^r(M)$. Using Theorem \ref{recurrence pour les structures} if $n\geq 3$ and Theorem \ref{extension surfaces} for surfaces, and using the previous observations, there exists a unique $G^r_\beta$-structure of order $r+1$ on $M$ extending $R_r$. 

This can be reformulated as follows. There exists a unique $G^r_\beta$-equivariant section 
$$
\Gamma_r:R_r\longrightarrow \mathcal R^{r+1}(M)
$$
of the natural map 
$$
\mathcal R^{r+1}(M) \longrightarrow \mathcal R^{r}(M).
$$
Notice that the uniqueness in the case of surfaces comes from the fact that in the proof of Theorem \ref{recurrence pour les structures}, the vanishing of the $h^0$ of the considered bundle is still true for surfaces.

We recall from Section \ref{subsection canonical form} the existence of the canonical form $\chi_{r+1}:T\mathcal R^{r+1}(M)\to T_0\mathbb C^n\times\mathfrak d^{r}(\mathbb C^n)$ on $\mathcal R^{r+1}(M)$. We consider $\eta_r$ the holomorphic $1$-form on $R_r$ with values in $T_0\mathbb C^n\times\mathfrak d^{r}(\mathbb C^n)$ defined by 
$$
\eta_r=\Gamma_r^*\chi_{r+1}.
$$
As we saw in Section \ref{subsection canonical form}, this exhibits the fact that $\Gamma_r$ is a $\mathcal R^r(M)$-horizontal plane field on $R_r$, and it generates a field of frames of order $1$ of $\mathcal R^r(M)$ over $R_r$. This implies that at each point $j^r_0f\in R_r$, the linear map $(\eta_r)_{j^r_0f}$ is one-to-one onto its image.

By the equivariance property of $\Gamma_r$, and by Proposition \ref{properties cano form}, the form $\eta_r$ satisfies the following properties
\begin{enumerate}
    \item $\forall v\in \mathfrak g^r_\beta$, $\eta_{r}(X_v)\equiv (0,v)$, where $X_v$ is the fundamental vector field of $R_r$ generated by $v$ ;
    \item $\forall P\in G^r_\beta$, $\mathcal{R}_{P}^*\eta_{r}=\underline{Ad}^{r+1}_{P^{-1}}\eta_{r}$.
\end{enumerate}

The main result of this subsection is the following.

\begin{Theorem}\label{on a une geometrie de Cartan}
    The form $\eta_r$ takes values in $T_0\mathbb C^n\times \mathfrak g^r_\beta$ and defines a holomorphic Cartan geometry of type $(\widetilde{G^r_\beta},G^r_\beta)$ on the principal bundle $R_r$.
\end{Theorem}
\begin{proof}
    We recall that we can consider every of our geometric objects as defined on $\mathbb C^n$ and with the property of being $<\gamma>$-equivariant. 

    According to the Lemma \ref{adjoint est bien adjoint}, the representation $\underline{Ad}^{r+1}\vert_{G^r_\beta}$ descends to a representation, denoted by $\underline{Ad}$ 
$$
G^r_\beta \longrightarrow GL((T_0\mathbb C^n\times \mathfrak d^{r}(\mathbb C^n))/(T_0\mathbb C^n\times \mathfrak g^r_\beta))
$$
and the form $\widetilde{\eta_r}:TR_r\to (T_0\mathbb C^n\times \mathfrak d^{r}(\mathbb C^n))/(T_0\mathbb C^n\times \mathfrak g^r_\beta)$, obtained by applying the standard quotient map to $\eta_r$, satisfies the following properties 
\begin{enumerate}
    \item $\forall v\in \mathfrak g^r_\beta$, $\widetilde {\eta_{r}}(X_v)\equiv 0$, where $X_v$ is the fundamental vector field of $R_r$ generated by $v$ ;
    \item $\forall P\in G^r_\beta$, $\mathcal{R}_{P}^*\widetilde{\eta_{r}}=\underline{Ad}^{r+1}_{P^{-1}}\widetilde{\eta_{r}}$.
\end{enumerate}
Notice that we have an isomorphism 
$$
T_0\mathbb C^n\times \mathfrak d^{r}(\mathbb C^n))/(T_0\mathbb C^n\times \mathfrak g^r_\beta)\cong \mathfrak d^{r}(\mathbb C^n)/\mathfrak g^r_\beta.
$$
Therefore, $\widetilde{\eta_r}$ can be seen as a holomorphic $<\gamma>$-invariant $1$-form on $\mathbb C^n$ with value in the associated vector bundle 
$$
R_r\times_{G^r_\beta}\mathfrak d^{r}(\mathbb C^n)/\mathfrak g^r_\beta
$$
where $G^r_\beta$ acts on $\mathfrak d^{r}(\mathbb C^n)/\mathfrak g^r_\beta$ via 
$$
\begin{matrix}
    G^r_\beta & \longrightarrow & GL(\mathfrak d^{r}(\mathbb C^n)) \\
    P & \longmapsto & \mathrm{d}_{j^r_0id}(j^r_0\phi\mapsto j^r_0(P\phi P^{-1})).
\end{matrix}
$$

We will show that such an invariant $1$-form on $\mathbb C^n$ vanishes.

Following the computations of Ornea and Verbitsky (\cite{OV2024}, 5.2. Resonant equivariant bundles p.16), in the proof of Theorem \ref{recurrence pour les structures} and the computations in the proofs of Section \ref{geo str on surf}, we need to compute the eigenvalues of the action of $\gamma$ on the fiber of the bundle $R_r\times_{G^r_\beta}\mathfrak d^{r}(\mathbb C^n)/\mathfrak g^r_\beta$ over the origin. The typical fiber of this bundle is the space $\mathfrak d^{r}(\mathbb C^n)/\mathfrak g^r_\beta$, and in a well chosen isomorphism, the action of $\gamma$ is given by $\underline{Ad}^{r+1}_{j^{r+1}_0\gamma^{-1}}$. 

The homomorphisms $\rho_{r,s}:\mathcal D^r(\mathbb C^n)\to \mathcal D^s(\mathbb C^n)$ for $0\leq s\leq r$ gives a filtration by normal Lie subgroups of the group $\mathcal D^r(\mathbb C^n)$ 
$$
\{id\}=N^r_r\subset N^r_{r-1} \subset ... \subset N^r_{1} \subset N^r_0=\mathcal D^r(\mathbb C^n) 
$$
where $N^r_s=\ker(\rho_{r,s})$. This filtration by normal subgroups gives a filtration by adjoint invariant Lie sub-algebras 
$$
\{0\}=\mathfrak n^r_r\subset \mathfrak n^r_{r-1} \subset ... \subset \mathfrak n^r_{1} \subset \mathfrak n^r_0=\mathfrak d^r(\mathbb C^n).
$$
Therefore, as in the proof of Theorem \ref{recurrence pour les structures}, the eigenvalues of the linear map $\underline{Ad}^{r+1}_{j^{r+1}_0\gamma^{-1}}$ over the space $\mathfrak d^{r}(\mathbb C^n)/\mathfrak g^r_\beta$ are 
$$
\beta_j^{-1}\beta^p
$$
with $j\in\{1,...,n\}$, $p\in \mathbb N^n$ of length satisfying $1\leq\vert p \vert \leq r$ and $p$ is not $j$ sub-resonant.

Hence, as usual, by taking Taylor expansion at $0$ of a $\gamma$-invariant holomorphic $1$-form with value in the bundle $R_r\times_{G^r_\beta}\mathfrak d^{r}(\mathbb C^n)/\mathfrak g^r_\beta$, we see that if the section is non identically zero, then there is a relation of the form 
$$
\beta^k\beta_l\beta_j^{-1}\beta^p=1
$$
where $k$ is an element of $\mathbb N^n$, $l$ is an element of $\{1,...,n\}$ and $j$ and $p$ are as above. But the existence of such a relation is a contradiction, because it implies that $p$ is $j$ sub-resonant.

Therefore, our form $\widetilde{\eta_r}$ is identically zero. We infer the first statement in the theorem, namely, that $\eta_r$ takes values in $T_0\mathbb C^n\times \mathfrak g^r_\beta$.

According to our previous observations, at each point of $R_r$, the form $\eta_r$ is one-to-one onto its image, so if we consider it as a form with value in $T_0\mathbb C^n\times \mathfrak g^r_\beta$, by an argument of dimensions, $\eta_r$ is an absolute parallelism. Moreover, we already know that the fundamental vector fields of $R_r$ are $\eta_r$ constant. The equivariance property of $\eta_r$ required in the definition of a Cartan geometry is given by Lemma \ref{adjoint est bien adjoint}.
\end{proof}

This result is equivalent to the following fact. The map $\Gamma_r:R_r\to \mathcal R^{r+1}(M)$ is a holomorphic horizontal plane field of the bundle $R_r\to M$. Notice that this gives rise to a connection on the principal bundle $R_r\to M$ but that this connection is not principal (except if $r=1$, in which case this connection induces the unique affine connection exhibited by Ornea and Verbitsky in \cite{OV2024}, p.27 Theorem 6.6).

\subsection{Integrability of the G-structure of order r+1}\label{integrability}

Let $R_r$ be a $G^r_\beta$-structure of order $r$ on $M$, and $\Gamma_r:R_r\to\mathcal{R}^{r+1}(M)$ be the $G^r_\beta$-equivariant section given by Theorem \ref{recurrence pour les structures} if $n\geq3$, and Theorem \ref{extension surfaces} if $n=2$. 

Applying successively Theorem \ref{recurrence pour les structures}, or Theorem \ref{extension surfaces}, we get also uniquely determined $G^r_\beta$ equivariant sections
$$
\Gamma_k:\im(\Gamma_{k-1})\longrightarrow \mathcal R^{k+1}(M)
$$
for every $k\geq r+1$ if $n\geq 3$, for every $k=r+1$ if $n=2$. 
We denote by $R_k$ the image of $\Gamma_{k-1}$. From this, we can prove the following.

\begin{Proposition}
    The $G^r_\beta$-structure of order $r+1$ $R_{r+1}$ is rigid of order $r$, in the sense of rigidity for Gromov geometric structures (see Section \ref{def rigidity}).
\end{Proposition}

\begin{proof}
    Let $x$ be a point of the manifold $M$, and let $j^{r+1}_xf$ be a jet of order $r+1$ at $x$ of a germ of biholomorphism of $M$ at $x$, and fixing $x$. Assume, moreover, that this jet preserves the sub-bundle $R_{r+1}$. Then, by the equivariance property of a jet of a biholomorphism, $j^r_xf$ preserves $R_r$. 

    Let $j_0^{r+1}g$ be an element of $R_{r+1}\vert_x$. The jet $j^{r+1}_xf$ is completely determined by the jet $j^{r+1}_0(fg)$. Let $P$ be the polynomial map of $G^r_\beta$ such that $j^r_0(fg)=j^r_0(gP)$, then by the equivariance property of $\Gamma_r$, we have that 
    $$
    j^{r+1}_0(fg)=j^{r+1}_0(gP).
    $$
    Therefore, $j^{r+1}_xf$ is completely determined by $P$, which is itself completely determined by $j^r_xf$. 
    Hence, the structure $R_{r+1}$ is rigid at order $r$.
\end{proof}

We will now prove the integrability of the $G^r_\beta$-structure of order $r+1$. The strategy is to use the point of view of Benoist, and namely, to use Theorem \ref{Th Frobenius}.

\begin{Theorem}\label{th integrability}
    The $G^r_\beta$-structure of order $r+1$ $R_{r+1}$ is integrable. 
\end{Theorem}

\begin{proof}
     We will prove that the reduction $R_{r+1}$ is, roughly speaking, complete and consistent in the terminology of Benoist (\cite{Be97}, Section 3.2 p.12, see Definition \ref{def comp et cons}). The Frobenius type theorem of Benoist (\cite{Be97}, Section 3.2 p.12, see Theorem \ref{Th Frobenius}) will provide a germ of solution of the partial derivative relation $R_{r+1}$ at $0\in \mathbb C^n$.

     As we saw before, using Hartogs Theorem (the extensions of our geometric structures are sections of trivial affine bundles), every structure that we have can be seen as a geometric structure on $\mathbb C^n$, invariant under the action of the contraction $\gamma$, so we consider now $R_k$ as a $G^r_\beta$-structure of order $k$ on $\mathbb C^n$ invariant by $\gamma$, for every $r\leq k\leq r+2$ and 
    $$
    \Gamma_k:R_k \longrightarrow R_{k+1}
    $$
    for every $r\leq k\leq r+1$. Notice that we will only use the extension of order $2$ or $R_r$, so we do not need to distinguish between $n\geq 3$ and $n=2$.

    For every integer $k\geq 1$, we consider $J_k:=J^k(\mathbb C^n,\mathbb C^n)$ the space of jets of order $k$ of germs of holomorphic maps from $\mathbb C^n$ to itself. We also consider 
    $$
    E_{r+1}:=\{j^{r+1}_zf\in J_{r+1}~\vert~j^{r+1}_0(f\theta_z)\in R_{r+1}\}.
    $$
    We will show that the partial derivative relation $E_{r+1}$ is complete and consistent in the terminology of Benoist (\cite{Be97} and Definition \ref{def comp et cons}). 
    We denote by $p_{r+1}$ the standard projection 
    $$
    p_{r+1}:J_{r+1}\longrightarrow J_r.
    $$
    Then we will show that $p_{r+1}$, when restricted to $E_{r+1}$, is a biholomorphism onto its image. This, by definition, proves that $E_{r+1}$ is complete.  
    Suppose that $j^{r+1}_zf$ and $j^{r+1}_wg$ are elements of $E_{r+1}$ such that 
    $$
    j^{r}_zf=j^{r}_wg.
    $$
    Then $z=w$ and 
    $$
    j^r_0(f\theta_z)=j^r_0(g\theta_z).
    $$
    But since $j^{r+1}_zf$ and $j^{r+1}_wg$ are elements of $E_{r+1}$, then $j^{r}_0(f\theta_z)$ and $j^{r}_0(g\theta_z)$ are elements of $R_r$ and 
    $$
    \left\{
    \begin{matrix}
        \Gamma(j^r_0(f\theta_z)) & = & j^{r+1}_0(f\theta_z)\\
        \Gamma(j^r_0(g\theta_z)) & = & j^{r+1}_0(g\theta_z)
    \end{matrix}
    \right.
    $$
    Therefore, we have that 
    $$
    j^{r+1}_0(f\theta_z)=j^{r+1}_0(g\theta_z),
    $$
    which implies that 
    $$
    j^{r+1}_zf=j^{r+1}_zg.
    $$
    This proves that $p_{r+1}$ restricted to $E_{r+1}$ is one-to-one, hence a biholomorphism onto its image. We deduce that the partial derivative relation $E_{r+1}$ is complete.

    We will now show that $E_{r+1}$ is consistent. By definition, for any element $x$ of $E_{r+1}$, we need to find a local section of $E_{r+1}\to \mathbb C^n$ passing through $x$ and tangent at $x$ to an holonomic section, i.e. a section that is actually a jet of a holomorphic germ. 

    We consider the $G^r_\beta$ equivariant section  $\Gamma_{r+1}:R_{r+1}\to \mathcal R^{r+2}(M)$. Applying exactly the same proof as in Theorem \ref{on a une geometrie de Cartan}, we can see the map $\Gamma_{r+1}$ as a connection on the bundle $R_{r+1}\to M$. 

    We consider the map 
    $$
    \begin{matrix}
        H_{r+1} & : & E_{r+1} & \longrightarrow & J_{r+2} \\
        && j^{r+1}_zf & \longmapsto & j^{r+2}_z(\widetilde{f\theta_z}\theta_{-z})
    \end{matrix}
    $$
    where $\widetilde{f\theta_z}$ is a germ representing $\Gamma_{r+1}(j^{r+1}_0(f\theta_z))$.

    Since the map $\Gamma_{r+1}$ gives rise to a connection on the bundle $R_{r+1}\to M$, the map $H_{r+1}$ is a horizontal plane field for the bundle $J^{r+1}\to \mathbb C^n$ over $E_{r+1}$ which is tangent to $E_{r+1}$ and holonomic. At each point of $E_{r+1}$, we consider a local section tangent at the point to the horizontal plane given by $H_{r+1}$. This is a section that is tangent to a holonomic plane. 
    Therefore, the partial derivative relation $E_{r+1}$ is consistent.

    By Theorem \ref{Th Frobenius}, at each point of $E_{r+1}$, there exists a unique germ of solution passing through this point. 

    Let $p$ be an element of $R_{r+1}$, and let $f$ be a germ at $0$ of solution of $E_{r+1}$ such that $p=j^{r+1}_0f$. Then  by definition, the biholomorphic germ $f$ maps the translations of $\mathbb C^n$ into $R_{r+1}$, and gives rise to a local isomorphism between the standard $G^r_\beta$-structure or order $r+1$ of $\mathbb C^n$ and $R_{r+1}$.

    This concludes the proof that $R_{r+1}$ is integrable.
\end{proof}

As an immediate consequence, $R_r$ is also integrable. In the language of Cartan geometries, we have the following.

\begin{Corollary}\label{flatness of the cartan geometry}
    The Cartan geometry $\eta_r$ on the bundle $R_r$ (see Theorem \ref{on a une geometrie de Cartan}) is flat.
\end{Corollary}

\begin{proof}
    According to Theorem \ref{th integrability}, the $G^r_\beta$-structure of order $r+1$ $R_{r+1}$ is integrable. Let $U$ be an open set in $M$ and $V$ be an open set in $\mathbb C^n$ such that there exists an integrability chart $f:V\to U$ for $R_{r+1}$. It implies that $f$ is also an integrability chart for $R_r$ and that the jet map 
    $$
    j^rf:V\times \mathcal D^r(\mathbb C^n)\longrightarrow \mathcal R^r(M)
    $$
    induces a isomorphism of $G^r_\beta$-structures of order $r$, still denoted 
    $$
    j^rf:V\times G^r_\beta\longrightarrow  R_r.
    $$
    Notice that $V\times G^r_\beta$ can be seen as an open set of $\widetilde{G^r_\beta}$ (see Corollary \ref{biholo with the product}). We compute the pullback of $\eta_r$ by the map $j^rf$:
    $$
    \begin{matrix}
        (j^rf)^*\eta_r & = & (j^rf)^*(\Gamma_r^*\chi_{r+1}) \\
        &=& (\Gamma_r\circ j^rf)^*\chi_{r+1} \\
        &=& (j^{r+1}f) ^*\chi_{r+1}
    \end{matrix}
    $$
    where $j^{r+1}f$ is considered as a map $V\times G^r_\beta\longrightarrow R_{r+1}$ and the last equality holds because of the fact that $f$ is an integrability chart for $R_{r+1}$. 

    The canonical form on $\mathcal R^{r+1}(M)$ is pulled back by $j^{r+1}f$ on the canonical form on $\mathcal R^{r+1}(\mathbb C^n)$. But now according to Corollary \ref{biholo with the product}, the restriction of the canonical form of $\mathcal R^{r+1}(\mathbb C^n)$ to the group $\widetilde{G^r_\beta}$ coincides with the Maurer Cartan form on the latter. Therefore, the form $\eta_r$ is pulled back on the Maurer Cartan form of $\widetilde{G^r_\beta}$, meaning that $j^rf$ is an isomorphism of Cartan geometries. By Theorem \ref{meaning of flatness for cartan geometries}, $\eta_r$ is flat.
 \end{proof}

\subsection{Poincaré-Dulac theorem}\label{subsect Poincare Dulac}

We will now give a new proof of the Poincaré-Dulac Theorem in the context of global holomorphic contractions of $\mathbb C^n$ in $0$. 

\begin{Theorem}\label{th PoincareDulac}
    Let $\gamma$ be an invertible contraction of $\mathbb C^n$ in $0$, where $n\geq2$, let $\beta=(\beta_1,...,\beta_n)$ be the collection of eigenvalues of $d_0\gamma$, and let $r$ be the maximal length of the resonances of $\beta$. Then there exists a global biholomorphism $\phi:\mathbb C^n\to \mathbb C^n$ such that $\phi^{-1}\gamma \phi$ is an element of $G^r_\beta$.  
\end{Theorem}

\begin{proof}
     According to Theorem \ref{th integrability}, there exists an integrable $G^r_\beta$-structure on $\mathbb C^n$ invariant under the action of $\gamma$, denoted by $R_r$. Let $U,V$ be two neighborhoods of $0\in\mathbb C^n$ and $\phi:U\to V$ be an integrable chart of $R_r$ on $V$ mapping $0$ to $0$.
    Now, since $\gamma$ preserves $R_r$, $\gamma\circ \phi$ is also an integrable chart. Hence, for every $z$ close enough to $0$, $j_0^{r}(\phi\circ \theta_{\phi^{-1}\gamma\phi(z)})$ and $j_0^{r}(\gamma\circ\phi\circ \theta_z)$ are in the same $G^r_\beta$ orbit. We deduce that for every $z$ close enough to $0$, $j_0^{r}(\theta_{-\phi^{-1}\gamma\phi(z)}\circ \phi^{-1}\circ\gamma\circ\phi\circ\theta_z)$ is an element of $G^r_\beta$. But this implies that the germ $\phi^{-1}\circ\gamma\circ\phi$ is itself an element of $G^r_\beta$. Therefore, it extends to a polynomial map $P$ in $G^r_\beta$.

    To conclude the proof, since $\gamma$ is a contraction in $0$, $P$ is also a contraction in $0$ and the germ of biholomorphism $\phi^{-1}$ gives a biholomorphism between the Hopf manifolds $M$ and $(\mathbb C^n\setminus\{0\})/<P>$. Therefore, $\phi^{-1}$, and \textit{a fortiori} $\phi$ are globally defined biholomorphic maps.
\end{proof}

Notice that here, we produce new coordinates in which $\gamma$ reads as a sub-resonant polynomial map. The Poincaré-Dulac theorem is slightly stronger, in the sense that we can find a system of coordinates in which $\gamma$ reads as a resonant polynomial map. But, starting with Theorem \ref{th PoincareDulac} and following the algebraic part of the proof of the Poincaré-Dulac theorem, we can find a polynomial change of coordinates (actually an element of $G^r_\beta$) such that $\gamma$ reads like a resonant polynomial map. In other words, in the conjugacy class of $\gamma$ in $G^r_\beta$, there is a resonant polynomial contraction. One of these resonant polynomial contractions is what is called a Poincaré-Dulac normal form of the contraction.  

If $\gamma$ is a contraction of $\mathbb C^n$ at $0$, and $\gamma$ reads in the standard coordinates of $\mathbb C^n$ like an element of its associated group $G^r_\beta$, then $\gamma$ obviously preserves all the standard $G^k_\beta$-structures of order $k$ of $\mathbb C^n$. All of those structures descend to the Hopf manifold and are obviously integrable. We deduce the following corollary.

\begin{Corollary}\label{exist str integrables}
    Let $M$ be a Hopf manifold given by a contraction $\gamma$ that induces the set of eigenvalues $\beta=(\beta_1,...,\beta_n)$ and the maximal length $r$. Then 
    \begin{enumerate}
        \item there exist integrable $G^k_\beta$-structures of order $k$ on $M$ for every $k\geq 1$. Moreover, those structures can be chosen so that their pullback on $\mathbb C^n\setminus\{0\}$ extends to $\mathbb C^n$.
        \item there exists a flat Cartan geometry on $M$ with model $(\widetilde{G^r_\beta},G^r_\beta)$. Moreover, this structure can be chosen so that its pullback on $\mathbb C^n\setminus\{0\}$ extends to $\mathbb C^n$.
    \end{enumerate}
\end{Corollary}

\newpage

\Large
\noindent\textbf{Acknowledgments}
\normalsize

We are very grateful to Sorin Dumitrescu and Laurent Meersseman for their helpful support throughout the development of this work. We also thank Bertrand Deroin, Frank Loray and Paul Boureau for the interesting discussions around the subject.

\nocite{*}
    \footnotesize{\bibliographystyle{plain}\bibliography{main}}

\vspace{0.5cm}
\texttt{UNIVERSITÉ COTE D'AZUR, LJAD, FRANCE}

Email address: \href{mailto:matthieu.madera@univ-cotedazur.fr}{matthieu.madera@univ-cotedazur.fr}

\end{document}